\documentclass[12pt,leqno]{amsart}
\usepackage{amsmath, amssymb, amscd, amsfonts, xypic,   stmaryrd, turnstile, mathrsfs, eucal,color}

\definecolor{dullmagenta}{rgb}{0.4,0,0.4}

\definecolor{darkblue}{rgb}{0,0,0.4}

\usepackage[colorlinks=true, pdfstartview=FitV, linkcolor=red, citecolor=blue, urlcolor=darkblue]
{hyperref}

\newtheorem{theorem}{Theorem}[section]
\newtheorem{lemma}[theorem]{Lemma}
\newtheorem{proposition}[theorem]{Proposition}
\newtheorem{corollary}[theorem]{Corollary}

\theoremstyle{definition}
\newtheorem{definition}[theorem]{Definition}
\newtheorem{example}[theorem]{Example}
\newtheorem{remark}[theorem]{Remark}

\begin{document}

\title[$\Delta$-extensions]{ FCP $\Delta$-extensions of rings}

\author[G. Picavet and M. Picavet]{Gabriel Picavet and Martine Picavet-L'Hermitte}
\address{Math\'ematiques \\
8 Rue du Forez, 63670 - Le Cendre\\
 France}
\email{picavet.mathu (at) orange.fr}

\begin{abstract} We consider ring extensions, whose set of all subextensions is stable under the formation of sums, the so-called $\Delta$-extensions. An integrally closed extension has the $\Delta$-property if and only it is a Pr\"ufer extension. We then give characterizations of FCP $\Delta$-extensions, using the fact that for FCP extensions, it is enough to consider integral FCP extensions. We are able to give substantial results. In particular, our work can be applied to extensions of number field orders because they have the FCP property.

\end{abstract} 


\keywords  {FIP, FCP extension, $\Delta$-extension,  minimal extension, integral extension, support of a module,  lattice, 
 modular lattice,
Boolean lattice, pointwise minimal extension}

\maketitle

\section{Introduction and Notation}

In this paper, we  work inside the category of commutative and unital rings, whose    epimorphisms will be involved. If  $R\subseteq S$ is a (ring) extension,  $[R,S]$  denote the set of all $R$-subalgebras of $S$.  
 An extension $R\subseteq S$ is said to have FCP (or is called an FCP extension) if the poset $([R,S), \subseteq)$ is both Artinian and Noetherian, which is equivalent to each chain in $[R,S]$ is finite.
  
The so-called $\Delta$-extensions have been the subject of many  papers. The seminal paper on the subject was authored by Gilmer and Huckaba \cite{GH}. A ring extension $R\subset S$ is called a {\it $\Delta$-extension} if $T+U\in [R,S]$ for each $T,U\in[R,S]$ ({\it i.e.} $T+U=TU$) \cite[Definition, page 414]{GH}. Although the notion of $\Delta$-extensions originates in Commutative Algebra, the lattice properties of  $[R,S]$ associated to an extension $R \subseteq S$  bring a new point of view to their theory.
 We will  explain what we are aiming to show about them in some contexts that have not  been  considered yet.
 
 We  consider lattices of the following form.
    For an extension $R\subseteq S$, the poset $([R,S],\subseteq)$ is a {\it complete} lattice,  where the supremum of any non-void subset  is the compositum of its elements, which we call {\it product} from now on and denote by $\Pi$ when necessary, and the infimum of any non-void subset is the intersection of its elements. 
    
      As a general rule, an extension $R\subseteq S$ is said to have some property of lattices if $[R,S]$ has this property.
 We use lattice definitions and properties described in \cite{NO}. 
 
Any undefined  material
 is explained at the end of the section or in the next sections.
 
 A representative example of the use of lattices is given by the following result. 
    A catenarian ({\it i.e.} verifying the Jordan-H\"older condition) integral FCP extension $R\subset S$, with $t$-closure $T$, is a $\Delta$-extension if and only if $R\subseteq T$  and $T \subseteq S$ are $\Delta$-extensions.
    Note also that an infra-integral (integral, with isomorphic residual field extensions) FCP extension has the $\Delta$-property if and only if it is modular.
    
In case we are dealing with an integrally closed extension, a characterization is immediately given as follows. Such extensions are $\Delta$-extensions if and only if they are Pr\"ufer extensions (defined by Knebush and Zhang \cite{KZ}) or equivalently they are normal pairs. This result is often reproved by authors working in some particular contexts.
  
We mainly consider FCP $\Delta$-extensions. The FCP condition allows us to prove results by induction. FCP extensions are of the form $R\subseteq\overline R\subseteq S$, where $\overline R$ is the integral closure of $R$ in $S$ and $\overline R\subseteq S$ is Pr\"ufer \cite[Proposition 1.3]{Pic 5}. We show that these extensions are $\Delta$-extensions if and only if $R\subseteq\overline R$ is a $\Delta$-extension (Theorem \ref{9.7}). Therefore, we need only to consider integral FCP extensions. Interesting examples of integral FCP extensions are given by extensions of number field orders. We exhibit examples of such extensions, showing that everything is possible. Note also that an extension $R\subset R[t]$ where $t$ is either idempotent or nilpotent of index 2, and such that $R$ is a SPIR, is a $\Delta$-extension.

 Section 2 is  devoted  to some recalls and results on ring extensions  and their lattice properties. 
 
 The general properties of $\Delta$-extensions are given in Section 3. 
 
In Section 4, the main result is Theorem \ref{9.24}, where we give a characterization of $\Delta$-extensions using the canonical decomposition of an integral extension $R\subseteq S$ through the seminormalization and the $t$-closure of $R$ in $S$. Actually the $\Delta$-property of $R\subseteq S$ is equivalent to the $\Delta$-property of all the paths of the canonical decomposition, plus an extra lattice condition relative to some $B_2$-subextensions. The case of infra-integral extensions is specially considered, while length two extensions are strongly involved. For example, an infra-integral FCP extension of length two is a $\Delta$-extension.

The paper ends in Section 5 with Examples of $\Delta$-extensions. In particular, we consider Boolean extensions,  pointwise minimal extensions, extensions of the form $R\subset R^n$. These special cases allow to characterize more generally some $\Delta$-extensions.
 
We denote by $(R:S)$ the conductor of $R\subseteq S$, and by $\overline R$ the integral closure of $R$ in $S$.
 We set $]R,S[:=[R,S]\setminus\{R,S\}$ (with a similar definition for $[R,S[$ or $]R,S]$).      
  
The extension $R\subseteq S$ is said to have FIP (for the ``finitely many intermediate algebras property") or is an FIP extension if $[R,S]$ is finite. A {\it chain} of $R$-subalgebras of $S$ is a set of elements of $[R,S]$ that are pairwise comparable with respect to inclusion. We will say that $R\subseteq S$ is {\it chained} if $[R,S]$ is a chain. We  also say that the extension $R\subseteq S$ has FCP (resp.; FMC) (or is an FCP (resp.; FMC) extension) if each chain in $[R,S]$ is finite (resp.; there exists a maximal finite chain). 
   Clearly,  each extension that satisfies FIP must also satisfy FCP 
 and each extension that satisfies FCP must also satisfy FMC. 
Dobbs and the authors characterized FCP and FIP extensions \cite{DPP2}.
  
Our principal tool  are  the minimal (ring) extensions, a concept that was introduced by Ferrand-Olivier \cite{FO}. 
In our context, minimal extensions coincide with lattice atoms. They are completely known (see Section 2).  Recall that an extension $R\subset S$ is called {\it minimal} if $[R, S]=\{R,S\}$. 
  The key connection between the above ideas is that if $R\subseteq S$ has FCP, then any maximal (necessarily finite) chain $\mathcal C$ of $R$-subalgebras of $S$, $R=R_0\subset R_1\subset\cdots\subset R_{n-1}\subset R_n=S$, with {\it length} $\ell(\mathcal C):=n <\infty$, results from juxtaposing $n$ minimal extensions $R_i\subset R_{i+1},\ 0\leq i\leq n-1$. 
An FCP extension is finitely generated, and  (module) finite if integral.
For an FCP extension $R\subseteq S$, the {\it length} $\ell[R,S]$ of $[R,S]$ is the supremum of the lengths of chains of $R$-subalgebras of $S$. Notice  that   this length is finite
and   there 
 {\it does} exist some maximal chain of $R$-subalgebras of $S$ with length $\ell[R,S]$ \cite[Theorem 4.11]
{DPP3}.

The characteristic of a field $k$ is denoted by $\mathrm{c}(k)$. Finally,  $|X|$ is the cardinality of a set $X$, $\subset$ denotes proper inclusion and, for a positive integer $n$, we set $\mathbb{N}_n:=\{1,\ldots,n\}$.  

 \section {Recalls and results on ring extensions}
This section is devoted to two types of recalls: commutative rings and lattices.

\subsection{Rings and ring extensions}
  A {\it local} ring is here what is called elsewhere a quasi-local ring. As usual, Spec$(R)$ and Max$(R)$ are the set of prime and maximal ideals of a ring $R$.  
     For an extension $R\subseteq S$ and an ideal $I$ of $R$, we write $\mathrm{V}_S(I):=\{P\in\mathrm{Spec }(S)\mid I\subseteq P\}$. 
The support of an $R$-module $E$ is $\mathrm{Supp}_R(E):=\{P\in\mathrm{Spec }(R)\mid E_P\neq 0\}$, and $\mathrm{MSupp}_R(E):=\mathrm{Supp}_R(E)\cap\mathrm{Max}(R)$. 
 Note that  if $R\subseteq S$ is an FMC (or FCP) extension, then $|\mathrm{Supp}_R(S/R)|<\infty$ \cite[Corollary 3.2]{DPP2}.
If $E$ is an $R$-module, ${\mathrm L}_R(E)$ (also denoted 
${\mathrm L}(E)$) is its length as a module.

If $R\subseteq S$ is a ring extension and $P\in\mathrm{Spec}(R)$, then $S_P$ is both the localization $S_{R\setminus P}$ as a ring and the localization at $P$ of the $R$-module $S$. 
  We denote by $\kappa_R(P)$ the residual field $R_P/PR_P$ at $P$. 

The following notions and results are  deeply involved in the sequel. 

\begin{definition}\label{crucial 1}\cite[Definition 2.10]{Pic 7} An extension $R\subset S$ is called  {\it $M$-crucial} if 
 $\mathrm{Supp}(S/R)=\{M\}$. Such $M$
is  called the {\it crucial (maximal) ideal}  $\mathcal{C}(R,S)$ of $R\subset S$. 
\end{definition}

\begin{theorem}\label{crucial}\cite[Th\'eor\`eme 2.2]{FO} A minimal extension $R\subset S$ is either integral (module-finite) or a flat epimorphism and $|\mathrm{Supp}(S/R)|=1$. Moreover, if $\mathrm{Supp}(S/R)=\{M\}$, then $M$
is the  crucial (maximal) ideal of $R\subset S$ (such that  $R_P=S_P$ for all $P\in\mathrm{Spec}(R)\setminus\{ M\}$).
\end{theorem} 

Recall that an extension $R\subseteq S$ is called {\it Pr\"ufer} if $R\subseteq T$ is a flat epimorphism for each $T\in[R,S]$ (or equivalently, if $R\subseteq S$ is a normal pair) \cite[Theorem 5.2, page 47]{KZ}. 
 In \cite{Pic 5}, we called an extension which is a minimal flat epimorphism, a {\it Pr\"ufer minimal} extension. Three types of minimal integral extensions exist, characterized in the next theorem, (a consequence of  the fundamental lemma of Ferrand-Olivier), so that there are four types of minimal extensions, mutually exclusive.
 
\begin{theorem}\label{minimal} \cite [Theorems 2.2 and 2.3]{DPP2} Let $R\subset T$ be an extension and  $M:=(R: T)$. Then $R\subset T$ is minimal and finite if and only if $M\in\mathrm{Max}(R)$ and one of the following three conditions holds:

\noindent  {\bf inert case}: $M\in\mathrm{Max}(T)$ and $R/M\to T/M$ is a minimal field extension.

\noindent  {\bf decomposed case}: There exist $M_1,M_2\in\mathrm{Max}(T)$ such that $M= M _1\cap M_2$ and the natural maps $R/M\to T/M_1$ and $R/M\to T/M_2$ are both isomorphisms, or equivalently, there exists $q\in T\setminus R$ such that $T=R[q],\ q^2-q\in M$ and $Mq\subseteq M$.

\noindent  {\bf ramified case}: There exists $M'\in\mathrm{Max}(T)$ such that ${M'}^2 \subseteq M\subset M',\  [T/M:R/M]=2$, and the natural map $R/M\to T/M'$ is an isomorphism, or equivalently, there exists $q\in T\setminus R$ such that $T=R[q],\ q^2\in M$ and $Mq\subseteq M$.

In each of the above  cases, $M=\mathcal{C}(R,T)$.
\end{theorem}

The crucial ideals of  minimal subextensions of an FCP extension give many useful properties as we can see in the following. 

\begin{lemma} \label{1.12} \cite[Lemma 1.5]{Pic 6} Let $R\subset S$ be an  extension and $T,U\in[R,S]$ such that $R\subset T$ is  finite minimal   and $R\subset U$ is  Pr\"ufer minimal. Then, $\mathcal{C}(R,T)\neq\mathcal{C}(R,U)$, so that $R$ is not a local ring.
\end{lemma}

\begin{lemma}\label{1.13} (Crosswise exchange) \cite[Lemma 2.7]{DPP2} Let $R\subset S$ and $S\subset T$ be minimal extensions,  $M:=\mathcal{C}(R,S)$, $N:=\mathcal{C}(S,T)$ and $P:=N\cap R$ be such that $P\not\subseteq M$. Then there is $S' \in [R,T]$ such that $R\subset S'$ is minimal of the same type as $S\subset T$ and $P= \mathcal{C}(R,S')$; and $S'\subset T$ is minimal of the same type as $R\subset S$ with $MS'=\mathcal{C}(S',T)$. Moreover, $[R,T]=\{R,S,S',T\}$ and $R_Q=S_Q=S'_Q=T_Q$ for $Q\in \mathrm{Max}(R)\setminus \{M,P\}$.
\end{lemma}

\subsection{ Some special ring extensions}

Let $R\subset S$ be an extension and $\mathcal{C}:=\{T_i\}_{i\in\mathbb N_n}\subset]R,S[,\ n\geq 1$ be a finite chain. We say that  $R\subset S$ is {\it pinched} at $\mathcal{C}$ if $[R,S]=\cup_{i=0}^n[T_i,T_{i+1}]$, where $T_0:=R$ and $T_{n+1}:=S$, which means that any element of $[R,S]$ is comparable to the $T_i$'s. 

If $R\subset S$ is an extension, we say that $R$ is {\it unbranched} in $S$ if $\overline R$ is local. We also say that $R\subset S$ is unbranched. If $R\subset S$ is unbranched and FCP, then each $T\in [R,S]$ is a local ring \cite[Lemma 3.29]{Pic 11}. An extension $R\subset S$ is said {\it locally unbranched} if $R_M\subset S_M$ is unbranched for all $M\in\mathrm{MSupp}(S/R)$. In particular, for an FCP extension $R\subset S$, this is equivalent to $\mathrm{Spec}(\overline R)\to\mathrm{Spec}(R)$ is bijective. An extension is said {\it branched} if it is not unbranched. 
An extension $R\subset S$ is said {\it almost unbranched} if each $T\in[R,\overline R[$ is a local ring. Then unbranched implies almost unbranched.  

\begin{remark} \label{1.120} Let $R\subset S$ be a ring extension and $T\in[R,S]$. Then, it is easily seen that $T\cap\overline R$ (resp.; $T\overline R$) is the integral closure of $R$ in $T$ (resp.; of $T$ in $S$). Similar relations exist for the t-closure and the seninormalization. We warn the reader that these properties will be often used in this paper. 
\end{remark}
 
\begin{corollary} \label{1.121} An FCP almost unbranched extension is  pinched at $\overline R$.
\end{corollary}

\begin{proof} Assume that $R\subset S$ is not pinched at $\overline R$, so that there exists $T\in[R,S]\setminus[R,\overline R]\cup[\overline R,S]$. Set $U:=T\cap \overline R\in[R,\overline R]$. Since $R\subset S$ is 
 almost unbranched and FCP, then $U$ is a local ring when $U\neq\overline R$, that is $T\not\in[\overline R,S]$, which is satisfied. 
Because $T\not\in[R,\overline R]$, it follows that $U\neq T$. 
 Then, there exist $V\in[U,T]$ and $W\in[U,\overline R]$ such that $U\subset V$ is minimal Pr\"ufer and $U\subset W$ is minimal integral, a contradiction by Lemma  \ref{1.12}. 
\end{proof}

\begin{proposition}\label{1.131} Let $R\subset S$ be an FCP extension such that $\overline R\neq R,S$. Then, $R\subset S$ is pinched at $\overline R$ if and only if, for any $U\in[R,S]$ such that $U\subset\overline R$ is minimal, then $\mathrm{MSupp}_{\overline R}(S/\overline R)\subseteq \mathrm{V}_{\overline R}((U:\overline R))$.  \end{proposition}

\begin{proof} Assume that $R\subset S$ is pinched at $\overline R$, so that $[R,S]=[R,\overline R]\cup[\overline R,S]$. Let $U\in[R,S]$ be such that $U\subset\overline R$ is minimal and set $P:=(U:\overline R)\in\mathrm{Max}(U)$. Let $M\in\mathrm{MSupp}_{\overline R}(S/\overline R)$. According to \cite[Lemma 1.8]{Pic 6}, there exists $V\in[\overline R,S]$ such that $\overline R\subset V$ is minimal Pr\"ufer with $M=\mathcal{C}(\overline R,V)$. If $M\cap U\not\subseteq P$, by the Crosswise Exchange, there exists $T\in[R,S]$ such that $U\subset T$ is minimal Pr\"ufer, so that $T\not\in [R,\overline R]$. Then, $T\in]\overline R,S]$, a contradiction with $U\subset \overline R\subset T$ and $U\subset T$ minimal. Then, $M\cap U\subseteq P$, and, more precisely, $M\cap U=P$ because $M\in\mathrm{Max}(\overline R)$ and $U\subset \overline R$ is integral. To conclude, $P\subseteq M$, that is $M\in \mathrm{V}_{\overline R}((U:\overline R))$.

Conversely, assume that $\mathrm{MSupp}_{\overline R}(S/\overline R)\subseteq\mathrm{V}_{\overline R}((U:\overline R))$ for any $U\in[R,S]$ such that $U\subset\overline R$ is minimal. Supppose that $[R,S]\neq[R,\overline R]\cup[\overline R,S]$ and let $T\in[R,S]\setminus([R,\overline R]\cup[\overline R,S])$. Set $U:=T\cap\overline R\subset\overline R$ because $U=\overline R $ implies $T\in[\overline R,S]$. We also have $U\neq T$ because $U=T$ implies $T\in[R,\overline R]$. Then, there exist $U_1\in[U,\overline R]$ and $T_1\in[U,T]$ such that $U\subset U_1$ is minimal integral and $U\subset T_1$ is minimal Pr\"ufer. By Lemma \ref{1.12}, we have $\mathcal{C}(U,U_1)\neq\mathcal{C}(U,T_1)$. It follows from \cite[Proposition 7.10]{DPPS} that $U_1\subset U_1T_1$ is minimal Pr\"ufer and $T_1\subset T_1U_1$ is minimal integral with $\mathcal{C}(U_1,U_1T_1)\not\in\mathrm{MSupp}_{U_1}(\overline R/U_1)$.  Of course, $U_1T_1\not\in[R,\overline R]$ because $T_1\in[U,U_1T_1]$. 

If $T_1U_1\in[\overline R,S]$, then $U_1\subseteq\overline R\subset T_1U_1$ implies $U_1=\overline R$, so that $U\subset\overline R$ is minimal. 

If $T_1U_1\not\in[\overline R,S]$, then $T_1\in[R,S]\setminus([R,\overline R]\cup[\overline R,S])$. It follows that we get $U_1=U_1T_1\cap\overline R$, because $U_1$ is the integral closure of $U\subseteq U_1T_1$. Since $\ell[U_1,\overline R]<\ell[U,\overline R]$, an easy induction shows that there exists a maximal finite chain $\{U_i\}_{i=0}^n$ such that $U_0=U,\ U_n=\overline R$ with $T_1U_i\not\in[R,\overline R]\cup[\overline R,S]$ for each $i\in\{0,\ldots,n-1\}$. We have the following commutative diagram:
$$\begin{matrix}
          T       &  {}  &      {}       &  \longrightarrow  & {}         & {}  &  S       \\
    \uparrow & {}  &      {}       &  {}  & {}                 & {}  & \uparrow          \\
  T_1          & \to & T_1U_1  & \to & T_1U_{n-1} & \to & T_1\overline R \\
   \uparrow & {}  & \uparrow & {}  & \uparrow      & {}   & \uparrow           \\
    U           & \to &      U_1   & \to & U_{n-1}       & \to & \overline R  
\end{matrix}$$
Then, in both cases, we get that $U_{n-1}\subset \overline R$ is minimal integral, with $U_{n-1}\subset T_1U_{n-1}$ minimal Pr\"ufer as $\overline R\subset T_1 \overline R$, so that $T_1\overline R\in[\overline R,S]$. We still have $\mathcal{C}(U_{n-1},U_{n-1}T_1)\neq\mathcal{C}(U_{n-1},\overline R)\ (*)$ with $\mathcal{C}(\overline R,\overline RT_1)\cap U_{n-1}=\mathcal{C}(U_{n-1},U_{n-1}T_1)\neq\mathcal{C}(U_{n-1},\overline R)=(U_{n-1}:\overline R)$. But $M:=\mathcal{C}(\overline R,\overline RT_1)\in\mathrm{MSupp}_{\overline R}(S/\overline R)$ gives by assumption that $(U_{n-1}:\overline R)\subseteq M$, so that $M\cap U_{n-1}=(U_{n-1}:\overline R)$ because $(U_{n-1}:\overline R)\in\mathrm{Max}(U_{n-1})$, a contradiction with $(*)$, which is $M\cap U_{n-1}\neq(U_{n-1}:\overline R)$. Hence, $[R,S]=[R,\overline R]\cup[\overline R,S]$.
\end{proof}

The following definitions are needed for our study.
  
\begin{definition}\label{1.3} An integral extension $R\subseteq S$ is called {\it infra-integral} \cite{Pic 2} (resp.; {\it subintegral} \cite{S}) if all its residual extensions $\kappa_R(P)\to \kappa_S(Q)$, (with $Q\in\mathrm {Spec}(S)$ and $P:=Q\cap R$) are isomorphisms (resp$.$; and the natural map $\mathrm {Spec}(S)\to\mathrm{Spec}(R)$ is bijective). An extension $R\subseteq S$ is called {\it t-closed} (cf. \cite{Pic 2}) if the relations $b\in S,\ r\in R,\ b^2-rb\in R,\ b^3-rb^2\in R$ imply $b\in R$. The $t$-{\it closure} ${}_S^tR$ of $R$ in $S$ is the smallest element $B\in [R,S]$  such that $B\subseteq S$ is t-closed and the greatest element  $B'\in [R,S]$ such that $R\subseteq B'$ is infra-integral. An extension $R\subseteq S$ is called {\it seminormal} (cf. \cite{S}) if the relations $b\in S,\ b^2\in R,\ b^3\in R$ imply $b\in R$. The {\it seminormalization} ${}_S^+R$ of $R$ in $S$ is the smallest element $B\in [R,S]$ such that $B\subseteq S$ is seminormal and the greatest  element $ B'\in[R,S]$ such that $R\subseteq B'$ is subintegral. 
  The {\it canonical decomposition} of an arbitrary ring extension $R\subset S$ is $R \subseteq {}_S^+R\subseteq {}_S^tR \subseteq \overline R \subseteq S$. 
 \end{definition}   
 
\begin{proposition}\label{3.5} \cite[Proposition, 4.5]{Pic 6} and \cite[Lemma 3.1]{Pic 4} Let there be an integral extension $R\subset S$ admitting a maximal chain $\mathcal C$ of $ R$-subextensions of $S$, defined by $R=R_0\subset\cdots\subset R_i\subset\cdots\subset R _n= S$, where each $R_i\subset R_{i+1}$ is  minimal.  The following statements hold: 

\begin{enumerate}
\item $R\subset S$ is subintegral if and only if each $R_i\subset R_{i+1}$ is  ramified. 

\item $R\subset S$ is  infra-integral if and only if each $R_i\subset R_{i+1}$ is either ramified or  decomposed.

\item $R\subset S$ is seminormal and infra-integral if and only if each  $R_i\subset R_{i+1}$ is decomposed. 

\item $R \subset S$ is t-closed if and only if  each $R_i\subset R_{i+1}$ is inert. 
\end{enumerate}

Moreover, $\mathrm{Spec}(S)\to\mathrm{Spec}(R)$ is bijective if and only if each $R_i\subset R_{i+1}$ is either ramified or inert.
\end{proposition}

\begin{proof} The last result comes from Theorem \ref{minimal} where  it is shown that decomposed minimal extensions are the only extensions whose spectral maps are not bijective.
\end{proof}

\subsection{Lattice Properties}
Let $R\subseteq S$ be an FCP extension, then $[R,S]$ is a complete Noetherian Artinian lattice, $R$ being the least element and $S$  the largest. In the context of the lattice $[R,S]$, some definitions and properties of lattices have the following formulations. (see \cite{NO})
 
(1) An element $T\in[R,S]$ is an {\it atom} if and only if $R\subset T$ is a minimal extension. We denote by $\mathcal{A}$  the set of atoms  of $[R,S]$. 

(2) $R\subseteq S$ is called {\it catenarian}, or graded by some authors working in the lattices context, if $R\subset S$ has FCP and all maximal chains between two comparable elements have the same length.
 
(3) $R\subseteq S$ is called {\it semimodular} if, for each $T_1,T_2\in[R,S]$ such that $T_1\cap T_2\subset T_i$ is minimal for $i=1,2$, then $T_i\subset T_1T_2$ is minimal for $i=1,2$.

(4) $R\subseteq S$ is called {\it modular} if $T_1\cap(T_2T_3)= T_2(T_1\cap T_3)$ for each $T_1,T_2,T_3\in[R,S]$ such that $T_2\subseteq T_1$. 

(5) $R\subseteq S$ is called {\it distributive} if intersection and product are each distributive with respect to the other. Actually, each distributivity implies the other \cite[Exercise 5, page 33]{NO}. 

Moreover, $R\subseteq S$ distributive $\Rightarrow R\subseteq S$ modular $\Rightarrow R\subseteq S$ semi-modular $\Rightarrow R\subseteq S$ catenarian.

(6) Let $T\in[R,S]$. Then, $T'\in[R,S]$ is called a {\it complement} of $T$ if $T\cap T'=R$ and $TT'=S$. 

(7) $R\subseteq S$ is called {\it Boolean} if $([R,S],\cap,\cdot)$ is a  distributive lattice such that each $T\in[R,S]$ has a (necessarily unique) complement.
 
(8) $R\subset S$ is called a $B_2$-{\it extension} if $R\subset S$ is a Boolean extension of length 2 (which is equivalent to $\ell[R,S]=2$ and $|[R,S]|=4$). See \cite[Fig. 1]{LSi}. In particular, if $R\subset S \subset T$ is an extension satisfying the Crosswise Exchange, then $R\subset T$ is a $B_2$-extension.

(9) $R\subseteq S$ is called {\it simple} if there exists $x\in S\setminus R$ such that $S=R[x]$.
 
(10) $R\subseteq S$ is called {\it arithmetic} if $R_P\subseteq S_P$ is  chained for each $P\in\mathrm{Spec}(R)$.
 
\begin{proposition} \label{1.0} \cite[Corollary 1.3.4, p. 172]{NO} A modular lattice of finite length is catenarian (the Jordan-H\"older chain condition holds).
\end{proposition} 

\begin{proposition}\label{cat} \cite[Proposition 4.7]{Pic 12} An infra-integral FCP extension is catenarian.
\end{proposition} 
   
According to \cite[Exercise 3.2, p. 37]{Cal}, a finite length lattice is modular if and only if it satisfies both {\it covering conditions}, which means that for a ring extension $R\subset S$, for each $T,U\in[R,S]$ such that $T\cap U\subset T$ is minimal, then $U\subset TU$ is minimal (upper covering condition), and for each $T,U\in[R,S]$ such that  $U\subset TU$ is minimal, then $T\cap U\subset T$ is minimal (lower covering condition). 
   In particular, we have the following:
   \begin{lemma} \label{1.1} Let $R\subset S$ be an FCP modular extension. Then, $\ell[T,UV]=2$ for $T, U, V\in[R,S],\ U\neq V$ such that $T\subset U$ and $T\subset V$ are minimal.
\end{lemma}

\begin{proof} Assume  that $R\subset S$ is modular and let $T, U, V\in[R,S],\ U\neq V$ such that $T\subset U$ and $T\subset V$ are minimal. Then, $T=U\cap V$ implies that $U\subset UV$ is minimal. Since any maximal chains of $R\subset S$ have the same length, $\ell[T,UV]=\ell[T,U]+\ell[U,UV]=2$.
\end{proof} 

\begin{proposition}\label{1.2} A length 2 extension is modular, whence catenarian.
\end{proposition}

\begin{proof} Let $R\subset S$ be a length 2 extension with $T_1,T_2,T_3\in[R,S]$ such that $T_2\subseteq T_1$. It $T_1=T_2$, then obviously $T_1\cap(T_2T_3)=T_2(T_1\cap T_3)$. If $T_1\neq T_2$, then either $T_2=R$ or $T_1=S$, and  $T_1\cap(T_2T_3)=T_2(T_1\cap T_3)$ also holds.
\end{proof} 

The following Proposition summarizes \cite[Propositions 7.1, 7.4, 7.6 and 7.10]{DPPS}. Different cases occur when considering  two minimal integral extensions with the same domain. The discussion is organized with respect to their crucial maximal ideals.
 
\begin{proposition}\label{3.6} Let $R\subset T$ and $R\subset U$ be two distinct minimal integral extensions, whose compositum $S:=TU$ exists. Set $M:=\mathcal{C}(R,T)$ and $N:=\mathcal{C}(R,U)$. The following statements hold: 
   \begin{enumerate}
 \item Assume that $M\neq N$. Then $[R,S]=\{R,T,U,S\}$.
 
\item Assume that $M=N,\ R\subset T$ is inert and $R\subset U$ is not inert. Then $R\subset S$ is not catenarian.
 
\item Assume that $M=N,\ R\subset T,\ R\subset U$ are both non-inert and $PQ\subseteq M$ for 
 some $P\in\mathrm{Max}(T)$ and some
 $Q\in\mathrm{Max}(U)$ both lying above $M$. Then $R\subset S$ is catenarian  infra-integral and $\ell[R,S]=2$. 

\item Assume that $M=N,\ R\subset T,\ R\subset U$ are both non-inert and $PQ\not\subseteq M$ for any $P\in\mathrm{Max}(T)$ and any
 $Q\in\mathrm{Max}(U)$ both lying above $M$. Then $R\subset S$ is  catenarian  infra-integral and  $\ell[R,S]=3$.  
\end{enumerate}
  \end{proposition}
  
\begin{proof}  We only prove that $[R,S]=\{R,T,U,S\}$ in (1). Indeed, \cite[Proposition 7.10]{DPPS} says that $T\subset S$ is a minimal extension such that $\mathcal{C}(T,S)$ lies over $N$ in $R$. In particular, $N=R\cap\mathcal{C}(T,S)\not\subseteq M$ shows that $|[R,S]|=4$ by the Crosswise Exchange (Lemma \ref{1.13}). Since $\{R,T,U,S\}\subseteq[R,S]$, the result follows.
  \end{proof} 

\section{General properties of $\Delta$-extensions}

We begin by recalling  a Gilmer-Huckaba's  result about $\Delta$-extensions.

\begin{proposition} \label{9.0} \cite[Proposition 1]{GH} Let $R\subset S$ be a ring extension. The following conditions are equivalent:
\begin{enumerate}
\item $R\subset S$ is a  $\Delta$-extension;
\item $R[s,t]=R[s]+R[t]$ for all $s,t\in S$;
\item $R[s_1,\dots,s_n]=\sum_{i=1}^nR[s_i]$ for all integer $n$ and   $s_1,\ldots,s_n\in S$.
\end{enumerate}
\end{proposition}

\begin{corollary} \label{9.2} A ring extension $R\subset S$ is  a  $\Delta$-extension if and only if $T\subset U$ is  a  $\Delta$-extension for each $T,U\in[R,S]$.
\end{corollary}

\begin{proof} Obvious.
\end{proof} 

Lemma \ref{1.1} and Proposition \ref{1.2} show that modular extensions and extensions of length 2 are linked. All along the paper, we will see that extensions of length 2 play a significant role. In particular, we have the following Corollary:

\begin{corollary} \label{9.04} A ring extension $R\subset S$ is not a $\Delta$-extension if there exist $T,U\in[R,S],\ T\subset U$,  such that $\ell[T,U]=2$ and  $T\subset U$ is not a  $\Delta$-extension.
\end{corollary}

\begin{proof} Obvious according to Corollary \ref{9.2}.
\end{proof} 

\begin{proposition} \label{9.1} A ring extension $R\subset S$ is  a  $\Delta$-extension if and only if $T+U=TU$ for each $T,U\in[R,S]$.
\end{proposition}

\begin{proof} Let $R\subset S$ be a $\Delta$-extension and let $T,U\in[R,S]$. Since  $T+U$ is an $R$-algebra, we get that $T,U\subseteq T+U\subseteq TU$, so that $T+U=TU$. Conversely, $T+U=TU$ shows that $T+U\in [R,S]$.
\end{proof} 

 \begin{corollary} \label{9.03} Let $R\subset S$ be a ring extension and  $\mathcal{C}:=\{T_i\}_{i\in\mathbb N_n}\subset]R,S[,\ n\geq 1$ be a finite chain such that  $R\subset S$ is  pinched at $\mathcal{C}$. Set $T_0:=R$ and $T_{n+1}:=S$. Then, $R\subset S$ is a $\Delta$-extension if and only if $T_i\subset T_{i+1}$ is a $\Delta$-extension for each $i\in\{0,\ldots,n\}$.
\end{corollary}
\begin{proof}
 One implication results from Corollary \ref{9.2}. Conversely, assume that $T_i\subset T_{i+1}$ is a $\Delta$-extension for each $i\in\{0,\ldots,n\}$ Let $U,V\in[R,S]=\cup_{i=0}^n[T_i, T_{i+1}]$. If $U,V$ are both  in some $[T_i, T_{i+1}]$, then $U+V=UV$ by Proposition \ref{9.1} applied  to $[T_i, T_{i+1}]$. Assume that $U\in[T_i, T_{i+1}]$ and $V\in[T_j, T_{j+1}]$ with $i\neq j$. Let, for instance, $i<j$, so that $i+1\leq j$. Then, $U\subseteq T_{i+1}\subseteq T_j\subseteq V$ leads to $U+V=V=UV$.  Another application of Proposition \ref{9.1} shows that $R\subset S$ is a $\Delta$-extension. 
\end{proof} 

\begin{proposition}\label{9.02} Let $R\subseteq S$ be a ring  extension. The following statements are equivalent: 
\begin{enumerate}
\item $R\subseteq  S$ is a $\Delta$-extension;

\item $R_M \subseteq S_M$ is  a $\Delta$-extension  for each $M\in\mathrm{MSupp}(S/R)$;

\item $R_P \subseteq S_P$ is  a $\Delta$-extension   for each $P\in\mathrm{Supp}(S/R)$;

\item  $R/I\subseteq S/I$ is  a $\Delta$-extension  for some  ideal $I$ shared by $R$ and $S$.
\end{enumerate}
\end{proposition}
\begin{proof} Obvious with the help of Proposition \ref{9.1}.
\end{proof} 

\begin{remark}\label{9.031} According to the equivalences of Proposition \ref{9.02}, and as we mostly consider in the rest of the paper FCP extensions, there is no harm to replace FCP integral extensions by locally FCP integral extensions. 
 Let $R\subset S$ be an integral extension satisfying the following property: any $T\in[R,S]$ such that $R\subset T$ is finite implies that $|\mathrm{MSupp}(T/R)|<\infty$. Then $R\subset S$ is locally FCP if and only if $R\subset S$ is locally finite, and for any $T\in[R,S]$ such that $R\subset T$ is finite, then $R\subset T$ has FMC.
  
First  assume that $R\subset S$ is locally finite and for any $T\in[R,S]$ such that $R\subset T$ is finite, then $R\subset T$ has FMC.
 Let $M\in\mathrm{Max}(R)$, so that there exist $x_1,\ldots,x_n\in R,\ s\in R\setminus M$ such that $S_M=R_M[x_1/s,\ldots,x_n/s]$. Set $T:=R[x_1,\ldots,x_n]$, which is a finite extension of $R$. In particular,  $R\subset T$ has FMC and so has $R_M\subseteq T_M=S_M$. It follows that  $R_M\subset S_M$ has FCP by \cite[Theorem 4.2]{DPP2}.

Conversely, assume that $R\subset S$ is locally FCP. The previous reference shows that $R_M\subset S_M$ is locally finite. Let $T\in[R,S]$ be such that $R\subseteq T$ is finite. In particular, $R_M\subseteq T_M$ is finite for each $M\in\mathrm{Max}(R)$ and $(R_M:T_M)=(R:T)_M$, so that $R_M/(R_M:T_M)=R_M/(R:T)_M=(R/(R:T))_M$. Since $R_M\subset S_M$ has FCP, so has $R_M\subseteq T_M$. But, $|\mathrm{MSupp}(T/R)|<\infty$ implies that $R\subset T$ has FCP  according to \cite[Proposition 3.7]{DPP2}, and then has FMC.
\end{remark} 

\begin{proposition} \label{desc} Let $R\subset S$  be a ring extension, $f: R \to R'$ a  ring morphism and $S':=  R'\otimes_RS$. 

\begin{enumerate}

\item If $f: R \to R'$ is a faithfully flat ring morphism and if $ R '\subset S'$ is a $\Delta$-extension,  then so is   $R\subset S$.
 
\item If $f: R \to R'$ is a flat ring epimorphism (for example, a localization with respect to a multiplicatively closed subset) and $ R \subset S$  is a $\Delta$-extension,  then so is   $R'\subset S'$.
\end{enumerate}
\end{proposition} 
\begin{proof} 

(1) The ring morphism $\varphi : S \to S'$ defines a map $\psi : [R,S] \to [R',S']$ by $\psi (T)= R'\otimes_RT$ and $ \theta: [R',S'] \to [R,S]$  by $\theta(T')=T'\cap S$ such that $ \theta \circ \psi   $ is the identity of $[R,S]$ by \cite[Proposition 10, p.52]{ALCO} (it is enough to take  $F=S$ and to observe that if $M$ is an $R$-submodule of $S$, then  with the notation of the above reference,  $R'M$ identifies to  $R'\otimes_RM$). The same reference shows that $\psi(T+U) = \psi(T) + \psi (U) $  and $\psi(T\cap U) = \psi(T) \cap  \psi (U) $ for $U, T\in [R,S]$. It is easy to show that $\psi(TU) = \psi(T)\psi(U)$. Then the result follows. 

(2) The proof is a consequence of the following facts. Let $f: R\to R'$ be a flat epimorphism and $Q\in\mathrm{Spec}(R')$, lying over $P$ in $R$, 
 then $R_P\to{R'}_Q$ is an isomorphism by \cite[Scholium A]{Pic 5}. 
 Moreover, we have $(R'\otimes_RS)_Q\cong R'_Q\otimes_{R_P}S_P$, so that $R_P\to S_P$ identifies to ${R'}_Q \to  (R'\otimes_RS)_Q $.
\end{proof}

\begin{remark} We deduce from the above statement (1), that the $\Delta$-property is local on the spectrum. This means that for a ring extension $R\subset S$ and any finite set $\{r_1,\ldots,r_n\}$ of elements of $R$, such that $R =Rr_1+\cdots+Rr_n$, then $R\subset S$ is a $\Delta$-extension if and only if all the extensions $R_{r_i} \subset S_{r_i}$ have the $\Delta$-property. 
\end{remark}

Given a ring $R$, recall that its {\it Nagata ring} $R(X)$ is the localization $R(X)=T^{-1}R[X]$ of the ring of polynomials $R[X]$ with respect to the multiplicatively closed subset $T$ of all  polynomials with  content $R$.
In \cite[Theorem 32]{DPP4}, Dobbs and the authors proved that when $R\subset S$ is an extension, whose Nagata extension $R(X)\subset S(X)$ has FIP, the map $\varphi:[R,S]\to[R(X), S(X)]$ defined by $\varphi(T)=T(X)$ is an order-isomorphism. We look at the transfer of the $\Delta$-property between $R\subset S$ and $R(X)\subset S(X)$.
 We recall the following \cite[Corollary 4.3]{Pic 4}: If $R\subset S$ has FIP, then $R(X)\subset S(X)$ has FIP if and  only if $R\subset {}_S^+R$ is arithmetic.

\begin{proposition}\label{8.14} An FCP extension $R\subset S$ is a $\Delta$-extension if  
$R(X)\subset S(X)$ is a $\Delta$-extension. 
 The converse hold if  $R(X)\subset S(X)$ has FIP. 
 \end{proposition}

\begin{proof}  Assume that $R(X)\subset S(X)$  is a $\Delta$-extension.
By \cite[Corollary 3.5]{DPP3}, we have $S(X)=R(X)\otimes_RS$. Since $R\subset R(X)$ is faithfully flat, an application of Proposition \ref{desc} gives the result.

Conversely, assume that $R\subset S$ is a $\Delta$-extension and $R(X)\subset S(X)$ has FIP. According to the previous remark, the map $\varphi:[R,S]\to [R(X),S(X)]$ defined by $\varphi(T)=T(X)$ is an order-isomorphism, and more precisely, a lattice isomorphism. Then, two elements of $[R(X),S(X)]$ are of the form $T(X),U(X)$ for some $T,U\in[R,S]$. In particular, $T(X)+U(X)=(T+U)(X)=(TU)(X)=T(X)U(X)$. Then, Proposition \ref{9.1} shows that $R(X)\subset S(X)$ is a $\Delta$-extension.
\end{proof} 

\begin{proposition} \label{1.013} An arithmetic extension is  a $\Delta$-extension. 
\end{proposition}

\begin{proof} \cite[Proposition 5.16]{Pic 4}.
\end{proof} 

\begin{corollary} \label{1.015} Let $R\subset S$ be a ring extension and $T,U\in[R,S]$ such that $R\subset T$ and $R\subset U$ are minimal. Assume that either $\mathcal{C}(R,T)\neq\mathcal{C}(R,U)$, or $R\subset T$ and $R\subset U$ are minimal of different types with $\ell[R,TU]=2$. 
 Then, $R\subset TU$ is a $B_2$-extension and a $\Delta$-extension.
\end{corollary}

\begin{proof} Set $M:=\mathcal{C}(R,T)$ and $N:=\mathcal{C}(R,U)$  
with $M\neq N$.
 Obviously, $\mathrm{Supp}(TU/R)=\mathrm{MSupp}(TU/R)=\{M,N\}$. Since $R_M=U_M$ and $R_N=T_N$, it follows that $R\subset TU$ is locally minimal,  then arithmetic and a $\Delta$-extension by Proposition \ref{1.013}. Moreover, \cite[Theorem 3.6]{DPP2} shows that the map $\varphi:[R,TU]\to [R_M,T_MU_M]\times[R_N,T_NU_N]$ defined by $V\mapsto (V_M,V_N)$ for any $V\in[R,TU]$, is bijective, with $T_MU_M=T_M$ and $T_NU_N=U_N$, so that $|[R,TU]|=4$, with $\ell[R,TU]=2$. Therefore,  $R\subset TU$ is a $B_2$-extension.
  
Assume now that $M:=\mathcal{C}(R,T)=\mathcal{C}(R,U)$ with $R\subset T$ and $R\subset U$  minimal of different types and $\ell[R,TU]=2$.  
In this case, neither $R\subset T$ nor $R\subset U$ is minimal Pr\"ufer by Lemma \ref{1.12}.
In particular, $R\subset TU$ is catenarian by Proposition \ref{1.2} and integral, so that neither $R\subset T$ nor $R\subset U$ is minimal inert according to Proposition \ref{3.6} (2).
   It follows that $R\subset TU$ is infra-integral with ${}_{TU}^+R\neq R,TU$. Assume that $R\subset T$ is minimal ramified and $R\subset U$ is minimal decomposed. Since $T\subset TU$ is minimal, it is necessarily decomposed, so that $T={}_{TU}^+R$. We now show that $[R,TU]=\{R,T,U,TU\}$. Assume that there exists some $W\in[R,TU]\setminus\{R,T,U,TU\}$, so that $R\subset W$ and $W\subset TU$ are minimal. Because $R\subset {}_{TU}^+R$ is minimal, we cannot have $R\subset W$ minimal ramified. But, $R\subset U$ and $R\subset W$ both minimal decomposed implies $R\subset UW$ seminormal infra-integral according to \cite[Proposition 7.6]{DPPS}, with $UW=S$, a contradiction. To conclude, $[R,TU]=\{R,T,U,TU\}$. 
    Then, $R\subset TU$ is a $B_2$-extension and a $\Delta$-extension.
   \end{proof}

\begin{proposition} \label{1.012}  A $\Delta$-extension is modular, and hence  is catenarian when it has  FCP.
\end{proposition}

\begin{proof} Let $R\subset S$ be a $\Delta$-extension and let $T,U,V\in[R,S]$ be such that $U\subseteq T$. We always have $U(T\cap V)\subseteq T\cap(UV)$. Since $T\cap(UV)=T\cap(U+V)\subseteq U+(T\cap V)=U(T\cap V)$, we get that $T\cap UV=U(T\cap V)$ and $R\subset S$ is a modular extension.
In particular, an FCP $\Delta$-extension is catenarian by Proposition \ref{1.0}. 
\end{proof} 

\begin{proposition} \label{1.014} \cite[Theorem 1]{GH} Let $k\subset L$ be a field extension. Then $k\subset L$ is  a $\Delta$-extension if and only if $k\subset L$ is chained. 
\end{proposition}

 In particular, the previous Proposition is satisfied when $k\subset L$ is either an FIP purely inseparable extension \cite[Lemma 4.1]{Pic 10}, or a cyclic extension whose degree is a power of a prime integer.

In \cite[before Proposition 2.7]{Pic 12}, we define the following notion:  A property $(\mathcal T)$ of ring extensions $R\subset S$ is called  {\it convenient} if the following conditions are equivalent.
\begin{enumerate}
\item $R\subset S$ satisfies $(\mathcal T)$. 

\item $R_M\subset S_M$ satisfies $(\mathcal T)$ for any $M\in\mathrm{MSupp}(S/R)$. 

\item $R/I\subset S/I$ satisfies $(\mathcal T)$ for any ideal $I$ shared by $R$ and $S$. 
\end{enumerate}

\begin{proposition}\label{5.13} Let $R\subseteq S$ be an FCP extension and $(\mathcal T)$ a convenient property. Assume that $R=\prod_{i=1}^n R_i$ is a product of rings. For each $i\in\mathbb N_n$, there exist FCP ring extensions $R_i\subseteq S_i$ such that $S\cong \prod_{i=1}^n S_i$. Moreover $R \subseteq S$ satisfies property $(\mathcal T)$ if and only if so does $R_i\subseteq S_i$ for each $i\in\mathbb N_n$. Since the $\Delta$-property is convenient, $R\subseteq S$ is a $\Delta$-extension if and only if so is $R_i\subseteq S_i$ for each $i\in\mathbb N_n$.
\end{proposition}
\begin{proof} The first part of the statement is \cite[Lemma III.3]{DMPP}, 
 except for the FCP property of the $R_i\subseteq S_i$. But this property is obvious since any $T\in[R,S]$ is of the form $\prod_{i=1}^n T_i$, where $T_i\in[R_i,S_i]$ for each $i\in\mathbb N_n$.

We recall the following of \cite[the last paragraph of Section 1]{Pic 9}:
 If $R_1,\ldots,R_n$ are finitely many rings,  the ring $R_1\times \cdots \times R_n$ localized at the prime ideal $P_1\times R_2\times\cdots \times R_n$ is isomorphic to $ (R_1)_{P_1}$ for $P_1 \in \mathrm{Spec}(R_1)$. This rule works for any prime ideal of the product. Since a maximal ideal $M\in\mathrm{Max}(R)$ is of the form $R_1\times \cdots \times M_i\times \cdots \times R_n$ for some $i\in\mathbb N_n$ and $M_i\in\mathrm{Max}(R_i)$, we get that $R_M\cong  (R_i)_{M_i}$ and $S_M\cong  (S_i)_{M_i}$. In particular, $R_M\subseteq S_M$ can be identified with $(R_i)_{M_i}\subseteq  (S_i)_{M_i}$. It follows that $M\in\mathrm{MSupp}(S/R)\Leftrightarrow R_M\neq S_M \Leftrightarrow (R_i)_{M_i}\neq  (S_i)_{M_i}\Leftrightarrow M_i\in\mathrm{MSupp}(S_i/R_i)$. Since  property $(\mathcal T)$ holds for $R\subseteq S$ if and only if it holds for $R_M\subseteq S_M$ for each $M\in\mathrm{MSupp}(S/R)$ because convenient, the previous isomorphisms give the last result.
\end{proof} 

\section {Characterization of $\Delta$-extensions} 

\subsection{First characterizations of $\Delta$-extensions}
Our results mainly hold for FCP extensions.

\begin{proposition} \label{9.3} Let $R\subset S$ be a ring extension. The following statements hold:
 \begin{enumerate}
\item $R\subset S$ is a Pr\"ufer extension if and only if $R\subset S$ is  an integrally closed   $\Delta$-extension.

\item If $R\subset S$ is  integrally closed and $|\mathrm{Supp}(S/R)|<\infty$,   the following conditions are equivalent:
 \begin{enumerate}
\item $R\subset S$ is a $\Delta$-extension;
\item $R\subset S$ is an FCP extension;
\item $R\subset S$ is an FIP extension.
\end{enumerate}
\end{enumerate}
\end{proposition}

\begin{proof} (1) A Pr\"ufer extension is integrally closed by 
\cite[Scholium B]{Pic 5}, then arithmetic by \cite[Theorem 5.17]{Pic 4} and a $\Delta$-extension by Proposition \ref{1.013}. 

Conversely, if $R\subset S$ is  an integrally closed   $\Delta$-extension, then $R\subset S$ is a Pr\"ufer extension by \cite[Theorem 1.7, page 88]{KZ}. 

(2) Using (1),  \cite[Proposition 1.3]{Pic 5} gives the equivalence of (a) and (b) because an integrally closed extension $R\subset S$ has FCP if and only if $R\subset S$ is Pr\"ufer with $|\mathrm{Supp}(S/R)|<\infty$. The equivalence of (b) and (c) follows from \cite[Theorem 6.3]{DPP2} because an integrally closed extension  has FIP if and only if it has FCP.  
\end{proof} 

\begin{lemma} \label{9.4} Let $R\subset S$ be a ring extension and $T,U\in[R,S]$. Then, $T+U=TU$ if $\mathrm{MSupp}(T/R)\cap \mathrm{MSupp}(U/R)=\emptyset$.
\end{lemma} 

\begin{proof} Let $M\in\mathrm{Max}(R)\setminus[\mathrm{MSupp}(T/R)\cup\mathrm{MSupp}(U/R)]$. Then, $R_M=T_M=U_M$ yields $(T+U)_M=T_M+U_M=R_M=T_MU_M=(TU)_M$.

Let $M\in\mathrm{MSupp}(T/R)$, so that $M\not\in\mathrm{MSupp}(U/R)$. Then, $R_M=U_M$ yields $(T+U)_M=T_M+U_M=T_M=T_MU_M=(TU)_M$.

Let $M\in\mathrm{MSupp}(U/R)$, so that $M\not\in\mathrm{MSupp}(T/R)$. Then, $R_M=T_M$ yields $(T+U)_M=T_M+U_M=U_M=T_MU_M=(TU)_M$.

Since  for each $M\in\mathrm{Max}(R)$, it holds that $(T+U)_M=(TU)_M$, we have $T+U=TU$.
\end{proof} 

In \cite[Definition 4.1]{Pic 5}, we call an extension $R\subset S$  {\it  almost-Pr\"ufer} if it can be factored $R\subseteq U\subseteq S$, where $R\subseteq U$ is Pr\"ufer and $U\subseteq S$ is integral. Actually, $U$ is the {\it Pr\"ufer hull} $\widetilde{R}$ of the extension, that is the greatest $T\in[R,S]$ such that $R\subseteq T$ is Pr\"ufer.

\begin{lemma} \label{9.5} Let $R\subset S$ be an FCP extension.  Then, $R\subseteq \overline R$ is  a  $\Delta$-extension if and only if $T\subseteq \overline T^{U}$ is  a  $\Delta$-extension for any subextension $T\subset U$ of $R\subset S$.
\end{lemma} 

\begin{proof} One implication is obvious. Now, assume that $R\subseteq \overline R$ is  a  $\Delta$-extension and consider the tower $R\subseteq T\subset U\subseteq S$. Set $R_1:=\overline R^T$ and $S_1:=\overline T^{U}$. We get the following commutative  diagram, where {\it i}, (resp. {\it p}) indicates an integral (resp. Pr\"ufer) extension:
$$\begin{matrix}
{} & {}  & T  & \overset{i}\rightarrow & S_1 & \overset{p}\rightarrow & U & {} \\
{} & {}  & p\uparrow & {} & p\uparrow & {} & \downarrow \\
R & \overset{i}\rightarrow &  R_1 & \overset{i}\rightarrow  &  \overline R^{U}   & \to          & S   
\end{matrix}$$
The definition of $R_1$ and $S_1$ implies  that $R_1\subseteq T$ and $\overline R^{U}\subseteq S_1=\overline T^{U}$ are  integrally closed, and then Pr\"ufer extensions by \cite[Proposition 1.3 (2)]{Pic 5}, while $T\subset S_1$ and $R_1\subset \overline R^{U}$ are  integral. Then, $R_1\subseteq S_1$ is almost-Pr\"ufer, with $T=\widetilde {R_1}^{S_1}$, the  Pr\"ufer hull of the extension $R_1\subseteq S_1$ and $\overline R^{U}=\overline {R_1}^{S_1}$. It follows that $\mathrm{MSupp}(T/R_1)\cap \mathrm{MSupp}(\overline R^{U}/R_1)=\emptyset$  and $\mathrm{MSupp}(S_1/R_1)=\mathrm{MSupp}(T/R_1)\cup\mathrm{MSupp}(\overline R^{U}/R_1)$ by \cite[Proposition 4.18]{Pic 5}. 

Since $\overline R^{U}\subseteq\overline R$, we deduce from Corollary  \ref{9.2} that $R_1\subseteq \overline R^{U}$ is a $\Delta$-extension, and so is $(R_1)_M\subseteq (\overline R^{U})_M$ for each $M\in\mathrm{Max}(R_1)$ by Proposition \ref{9.02}. Let $V,W\in[T,S_1]$ and $M\in\mathrm{Max}(R_1)$. If $M\not\in\mathrm{MSupp}(\overline R^{U}/R_1)$, we get that $(\overline R^{U})_M=(R_1)_M$, which implies that $(R_1)_M\subseteq(S_1)_M$ is Pr\"ufer, so that $T_M=(S_1)_M=V_M=W_M$. Then, $(V+W)_M=V_M+W_M=V_M=V_MW_M=(VW)_M$. If $M\in\mathrm{MSupp}(\overline R^{U}/R_1)$, we get that $M\not\in\mathrm{MSupp}(T/R_1)$, so that $T_M=(R_1)_M$, which in turn  implies that $(R_1)_M\subseteq(S_1)_M$ is integral and $(S_1)_M=(\overline R^{U})_M$. This shows that $T_M=(R_1)_M\subseteq (S_1)_M=(\overline R^{U})_M$ is a $\Delta$-extension, so that $(V+W)_M=V_M+W_M=V_MW_M=(VW)_M$. To conclude, we have $(V+W)_M=(VW)_M$ for each $M\in\mathrm{Max}(R_1)$, and then $V+W=VW$, showing that $T\subset S_1=\overline T^{U}$ is  a  $\Delta$-extension. 
\end{proof} 

\begin{proposition} \label{9.6} Let $R\subset S$ be an FCP extension and $T,U\in[R,S]$ such that $R\subseteq T$ is integral and $R\subseteq U$ is Pr\"ufer. Then $U+T=UT$.  
 \end{proposition}

\begin{proof}   Since $R\subseteq T$ is integral, we have $T\in[R, \overline R]$ and since $R\subseteq U$ is Pr\"ufer, we have $U\in[R, \tilde R]$ by definition of the Pr\"ufer hull. Then, $\mathrm{MSupp}(T/R)\subseteq \mathrm{MSupp}(\overline R/R)$ and $\mathrm{MSupp}(U/R)\subseteq \mathrm{MSupp}(\tilde R/R)$. But $\mathrm{MSupp}(\overline R/R)\cap \mathrm{MSupp}(\tilde R/R)=\emptyset$ by \cite[Proposition 4.18]{Pic 5}. It follows that $\mathrm{MSupp}(T/R)\cap \mathrm{MSupp}(U/R)=\emptyset$ and Lemma \ref{9.4} shows that $U+T=UT$.
\end{proof} 

\begin{theorem} \label{9.7} An FCP extension $R\subset S$ is  a  $\Delta$-extension if and only if $R\subseteq \overline R$ is  a  $\Delta$-extension.
\end{theorem}

\begin{proof} Corollary \ref{9.2} gives one implication. Conversely, assume that $R\subset S$ is  an FCP extension such that $R\subseteq \overline R$ is  a  $\Delta$-extension. Let  $V,W\in[R,S]$ and set $T:=V\cap W,\ U:=VW$. We denote by $ \overline U$ (resp. $\overline V,\ \overline W$) the integral closure of $T$ in $U$ (resp. $V,\ W$). According to Lemma \ref{9.5}, $T\subseteq \overline U$ is  a  $\Delta$-extension. In particular, $\overline V+ \overline W=\overline V \ \overline W$ (1) since $\overline V,\ \overline W\subseteq   \overline U$. 

We get the following commutative diagram, where {\it i}, (resp. {\it p}) indicates an integral (resp. Pr\"ufer) extension:
$$\begin{matrix}
{} & {} & \overline V & \overset{p}\rightarrow & V & \overset{i}\rightarrow & V\overline W & {}  & {} \\
{} & i\nearrow & {} & i\searrow & {} & p\nearrow & {} & p\searrow &{} \\
T & {} & {} & {} & \overline V \ \overline W & {} & {} & {} & U \\
{} & i\searrow & {} & i\nearrow & {} & p\searrow & {} & p\nearrow & {} \\
{} & {} & \overline W & \overset{p}\rightarrow & W & \overset{i}\rightarrow & \overline VW & {}  & {} 
\end{matrix}$$
Since $ \overline V\subseteq V$ is Pr\"ufer and $\overline V\subseteq \overline V\ \overline W$ is integral, Proposition \ref{9.6} leads to $V+\overline V\ \overline W=V\overline V\ \overline W=V\overline W$ (2). For the same reason, we get $W+\overline V\ \overline W=\overline V W$ (3).

Moreover, $\overline V\subseteq V\overline W$ and $\overline W\subseteq\overline VW$ are almost-Pr\"ufer with $V$ (resp. $W$) as Pr\"ufer hull of $\overline V\subseteq V\overline W$ (resp. $\overline W\subseteq\overline VW$). According to \cite[Proposition 4.16]{Pic 5}, $\mathrm{MSupp}(V/\overline V)\cap\mathrm{MSupp}(\overline V\ \overline W/\overline V)=\emptyset$.
We claim that, $\overline V\overline W\subseteq V\overline W$ and $\overline V\overline W\subseteq\overline VW$ are Pr\"ufer. 
Indeed, $\mathrm{MSupp}(V\overline W/\overline V)=\mathrm{MSupp}(V/\overline V)\cup\mathrm{MSupp}( V \overline W/ V)$. Let $M\in\mathrm{MSupp}(V\overline W/\overline V)$. If $M\in\mathrm{MSupp}(V/\overline V)$, then $ V_M=V_M\overline W_M$, so that $\overline V_M\overline W_M\subseteq V_M=V_M\overline W_M$ is Pr\"ufer. If $M\in\mathrm{MSupp}(V\overline W/ V)$, then $V_M =\overline V_M$ leads to $\overline V_M\overline W_M=V_M\overline W_M$. It follows that $\overline V\overline W\subset V\overline W$ is Pr\"ufer. The same holds for $\overline V\overline W\subset\overline VW$. 
It follows that $\overline V \overline W\subseteq(V\overline W)(\overline VW)=VW$ is Pr\"ufer since $V \overline W$ and $ \overline V  W$ are both contained in the Pr\"ufer hull of $\overline V \overline W\subset U$. Then, $\overline V \overline W\subseteq VW$ is a $\Delta$-extension by Proposition \ref{9.3}, so that $V\overline W+\overline VW=V\overline W\ \overline VW=VW$ (4).  

Adding (2) and (3) and using (1) and (4), this  leads to  $V+\overline V\ \overline W+W+\overline V\ \overline W=V+\overline V+ \overline W+W+\overline V+ \overline W=V+W=V\overline W+\overline V W=VW$. 

To conclude, we obtain $V+W=VW$ for any $V,W\in[R,S]$ and $R\subset S$ is a $\Delta$-extension.
\end{proof} 

\begin{remark} \label{9.71} When  $R\subset S$ is an almost unbranched FCP extension the result of Theorem  \ref{9.7} is gotten immediately. Indeed, Corollary \ref{1.121} shows that $R\subset S$ is pinched at $\overline R$. Since $\overline R\subseteq S$ is  a  $\Delta$-extension by Proposition \ref{9.3}, then Corollary \ref{9.03} gives the result of Theorem  \ref{9.7}.
\end{remark}

Theorem \ref{9.7} shows that it is enough to characterize integral $\Delta$-extensions that are FCP, which we are aiming to do by using the paths  of the canonical decomposition.

\begin{proposition} \label{9.10} A t-closed FCP extension $R\subset S$ is a $\Delta$-extension if and only if $R\subset S$ is arithmetic.
\end{proposition}

\begin{proof} One implication is Proposition \ref{1.013}. Conversely, assume that $R\subset S$ is a $\Delta$-extension. By Proposition \ref{9.02}, we can assume that $(R,M)$ is a local ring, so that $(R:S)=M$ because $(S,M)$ is a local ring by \cite[Lemma 3.17]{DPP3}. Moreover, $R/M\subset S/M$ is a $\Delta$-extension again by Proposition \ref{9.02} and a field extension. Then, $R/M\subseteq S/M$ is chained by Proposition \ref{1.014} and so is $[R,S]$.
\end{proof} 

\begin{lemma} \label{9.12} If $R\subset S$ is a catenarian  FCP extension and $T,U\in]R,S]$ are such that $R\subset T$ is infra-integral and $R\subset U$ is t-closed, then, $T+U=TU$ and $\mathrm{MSupp}(T/R)\cap \mathrm{MSupp}( U/R)=\emptyset$.
\end{lemma} 

\begin{proof}  We claim  that $\mathrm{MSupp}(T/R)\cap \mathrm{MSupp}( U/R)=\emptyset$. Deny and let $M\in\mathrm{MSupp}(T/R)\cap \mathrm{MSupp}( U/R)$. According to  \cite[Lemma 1.8]{Pic 6}, there exist $V\in[R,T]$ and $W\in[R,U]$ such that $R\subset V$ is minimal infra-integral (that is either ramified or decomposed) and $R\subset W$ is minimal inert, while both extensions have   $M$ as conductor. It follows that there exist two maximal chains from $R$ to $VW$ whose lengths are different by Proposition \ref{3.6},  contradicting that $R\subset S$ is catenarian. Moreover,  $\mathrm{MSupp}(TU/R)=\mathrm{MSupp}(T/R)\cup\mathrm{MSupp}(U/R)$. Now let $M\in\mathrm{MSupp}(TU/R)$. 

If $M\in\mathrm{MSupp}(T/R)$, then $M\not\in\mathrm{MSupp}(U/R)$, so that $R_M=U_M$, giving $(T+U)_M=T_M+U_M=T_M=T_MU_M=(TU)_M$. If $M\in\mathrm{MSupp}(U/R)$, then, $M\not\in\mathrm{MSupp}(T/R)$,  so that $R_M=T_M$, giving $(T+U)_M=T_M+U_M=U_M=T_MU_M=(TU)_M$.

If finally  $M\not\in\mathrm{MSupp}(TU/R)$, then $R_M=T_M=U_M$, giving $(T+U)_M=T_M+U_M=R_M=T_MU_M=(TU)_M$. 
 
Since $(T+U)_M=(TU)_M$ for any $M\in\mathrm{Max}(R)$, we get $T+U=TU$.
\end{proof} 

\begin{lemma} \label{9.11} Let $R\subset S$ be a catenarian integral FCP extension. Then, ${}_S^tR\subseteq S$ is arithmetic if and only if ${}_U^tT\subseteq U$ is arithmetic for any subextension $T\subset U$ of $R\subset S$. 
\end{lemma} 

\begin{proof} One implication is obvious. Now, assume that ${}_S^tR\subseteq S$ is arithmetic. Consider the tower $R\subseteq T\subset U\subseteq S$ and set $R_1:={}_U^tR,\ R_2:={}_S^tR$ and $T_1:={}_U^tT$. We get the following commutative  diagram, where {\it i}, (resp. {\it t}) indicates an infra-integral (resp. t-closed) extension:
$$\begin{matrix}
 T  & \overset{i}\rightarrow & T_1 & \overset{t}\rightarrow & U & {} & {} & {}Ê\\
 \uparrow & {} & t\uparrow & {} & {} & \searrow & {} \\
  R & \overset{i}\rightarrow  &  R_1   &  \overset{i}\rightarrow          & R_2  & \overset{t}\rightarrow & S 
\end{matrix}$$
Since $R_1$ is the t-closure of $R$ in $U$ and $R\subseteq T\subseteq U\subseteq S$, we get that $R_1\subseteq T_1$ and $R_1\subseteq R_2$.  The definition of $R_1, R_2$ and $T_1$ implies that $R_1\subseteq T_1,\ T_1\subseteq U$ and $R_2\subseteq S$ are  t-closed, while $R\subseteq R_1,\ T\subseteq T_1$ and $R_1\subseteq  R_2$ are  infra-integral.

By Lemma \ref{9.12} and because $R\subset S$ is catenarian,  
$\mathrm{MSupp}_{R_1}(U/R_1)\cap\mathrm{MSupp}_{R_1}(R_2/R_1)=\emptyset$. Moreover, $\mathrm{MSupp}_{R_1}(U/R_1)\cup\mathrm{MSupp}_{R_1}(R_2/R_1)$

\noindent $=\mathrm{MSupp}_{R_1}(UR_2/R_1)$.
 Let $N\in\mathrm{MSupp}_{T_1}(U/T_1)$ and set $M:=N\cap R_1$,
 so that $M\in\mathrm{MSupp}_{R_1}(U/R_1)$.
  According to \cite[Lemma 3.17]{DPP3}, $M(R_1)_M=((R_1)_M:(T_1)_M)=((R_1)_M:U_M)$ and is the maximal ideal of the local rings $(T_1)_M$ and $U_M$ because $(R_1)_M\subset U_M$ is t-closed as is $(R_1)_M\subset (T_1)_M$. In particular, $N$ is the only maximal ideal of $T_1$ lying above $M$ and we get $(T_1)_M=(T_1)_N$ and $U_M=U_N$. Moreover, $M\not\in\mathrm{MSupp}(R_2/R_1)$ which yields $(R_1)_M=(R_2)_M$. To conclude, we get the tower $(R_2)_M=(R_1)_M\subseteq (T_1)_M\subset U_M\subseteq S_M$, where $(R_2)_M\subseteq  S_M$ is chained, and so is $(T_1)_M= (T_1)_N\subset U_M= U_N$. Then, ${}_U^tT\subseteq U$ is arithmetic.
\end{proof} 

\begin{lemma} \label{9.14} Let $R\subset S$ be a catenarian integral FCP extension. Then, $R\subseteq{}_S^tR$ is a $\Delta$-extension  if and only if $T\subseteq{}_U^tT $ is a $\Delta$-extension  for any subextension $T\subset U$ of $R\subset S$.
\end{lemma} 

\begin{proof} One implication is obvious. Now, assume that $R\subseteq {}_S^tR$ is a $\Delta$-extension. Consider the tower $R\subseteq T\subset U\subseteq S$. Set $R_1:={}_U^tR,\ R_2:={}_S^tR,\ T_1:={}_T^tR$ and $T_2:={}_U^tT$. There is no harm to assume that $T_2\neq T$. We get the following commutative diagram, where {\it i}, (resp. {\it t}) indicates an infra-integral (resp. t-closed) extension:

$$\begin{matrix}
T_1 & \overset{t}\rightarrow & T & \overset{i}\rightarrow & T_2 & \overset{t}\rightarrow & U Ê\\
i\uparrow & i\searrow & {} & t\nearrow & {} & {} & \downarrow \\
R & \overset{i}\rightarrow & R_1 & \overset{i}\rightarrow          & R_2 & \overset{t}\rightarrow & S 
\end{matrix}$$

Since $R_1$ is the t-closure of $R$ in $U$ and $R\subseteq T\subseteq U\subseteq S$, we get that $T_1\subseteq R_1\subseteq T_2$ and $R_1\subseteq R_2$. Then, $R_1$ is also the t-closure of $T_1$ in $T_2$. It follows from the definitions of $R_1, R_2,T_1$ and $T_2$ that $T_1\subseteq T,\ R_1\subseteq T_2,\ T_2\subseteq U$ and $R_2\subseteq S$ are  t-closed, while $R\subseteq T_1,\ T_1\subseteq R_1,\ R\subseteq R_1,\ T\subseteq T_2$ and $R_1\subseteq  R_2$ are  infra-integral. Since $R\subseteq R_2$ is a $\Delta$-extension, so is $T_1\subseteq R_1$. Let $N\in\mathrm{MSupp}_T(T_2/T)$ and set $M:=N\cap T_1\in\mathrm{Max}(T_1)$. Then $M\in\mathrm{MSupp}_{T_1}(T_2/T)$ and $(T_1)_M\subset (T_2)_M$ is not t-closed, 
 whence $(T_1)_M\neq(R_1)_M$, from whch we deduce that $M\in\mathrm{MSupp}_{T_1}(R_1/T_1)$. Now  Lemma \ref{9.12} shows that
  $M\not\in\mathrm{MSupp}_{T_1}(T/T_1)$. In particular, $N$ is the only maximal ideal of $T$ lying above $M$. Then, we have $(T_1)_M=T_M=T_N$ by \cite[Lemma 2.4]{DPP2}. For the same reason, $(T_2)_M=(T_2)_N=(R_1)_M$, where the last equality is consequence of the following facts: $(T_1)_M\subset(T_2)_M$ is infra-integral and $R_1\subset T_2$ is t-closed with $R_1\in[T_1,T_2]$. Then, $(T_1)_M=T_M=T_N\subset(T_2)_N=(T_2)_M=(R_1)_M$ is a $\Delta$-extension because so is $R\subseteq R_1$. Since this property holds for any $N\in\mathrm{MSupp}_T(T_2/T)$, we get that $T\subseteq T_2={}_U^tT $ is a $\Delta$-extension by Proposition \ref{9.02}.
\end{proof} 
 
The proof of the next theorem mimics the proof of Theorem~\ref{9.7}.
 
\begin{theorem} \label{9.15} A catenarian integral FCP extension $R\subset S$ is a  $\Delta$-extension if and only if $R\subseteq {}_S^tR$  and ${}_S^tR\subseteq S$ are $\Delta$-extensions.
 The last condition is equivalent to ${}_S^tR\subseteq S$ is arithmetic. 
\end{theorem} 

\begin{proof} One implication is obvious because of Proposition \ref{9.10} and Corollary \ref{9.2}. Conversely, assume that $R\subset S$ is a catenarian FCP integral extension such that $R\subseteq{}_S^tR$ is a $\Delta$-extension and ${}_S^tR\subseteq S$ is arithmetic. Let $V,W\in[R,S]$ and set $T:=V\cap W,\ U:=VW$ and $U_1:={}_U^tT$ (resp. $V_1:={}_V^tT,\ W_1:={}_W^tT$). 
 
According to Lemma \ref{9.14}, we get that $T\subseteq U_1$ is a $\Delta$-extension. In particular, $V_1+W_1=V_1 W_1$ (1) since 
$ V_1,\  W_1\in[T,   U_1]$. 

We get the following commutative  diagram, where {\it i}, (resp. {\it t}) indicates an infra-integral (resp. t-closed) extension:

$$\begin{matrix}
{} & {} & V_1 & \overset{t}\rightarrow & V & \overset{i}\rightarrow & VW_1 & {}  & {} \\
{} & i\nearrow & {} & i\searrow & {} & t\nearrow & {} & t\searrow &{} \\
T & {} & {} & {} & V_1 W_1 & {} & {} & {} & U=VW \\
{} & i\searrow & {} & i\nearrow & {} & t\searrow & {} & t\nearrow & {} \\
{} & {} & W_1 & \overset{t}\rightarrow & W & \overset{i}\rightarrow & V_1W & {}  & {} 
\end{matrix}$$

Since $V_1\subseteq V$ is t-closed and $V_1\subseteq V_1W_1$ is infra-integral because $T\subseteq V_1,W_1\subseteq U_1$ implies $T\subseteq V_1W_1\subseteq U_1$, the t-closure of $T$ in $U$, so that $T\subseteq U_1$ is infra-integral, Lemma \ref{9.12} gives $V+ V_1W_1=VV_1W_1=VW_1$ (2). For the same reason, we get $W+V_1W_1=V_1 W$ (3). Now, $V_1W_1\subseteq VW_1$ and $V_1W_1\subseteq V_1 W$ are t-closed. Indeed, Lemma \ref{9.12} yields $\mathrm{MSupp}_{V_1}(V/V_1)\cap\mathrm{MSupp}_{V_1}(V_1W_1/ V_1)=\emptyset$ and $\mathrm{MSupp}_{V_1}(V/V_1)\cup\mathrm{MSupp}_{V_1}(V_1 W_1/V_1)=\mathrm{MSupp}_{V_1}(VW_1/V_1)$. Let $M\in\mathrm{MSupp}_{V_1}(VW_1/V_1)$. 

If $M\in\mathrm{MSupp}_{V_1}(V/V_1)$, then $(V_1)_M=(V_1)_M(W_1)_M$ entails that $(V_1)_M=(V_1)_M(W_1)_M\subset V_M(W_1)_M=
(V_1)_MV_M(W_1)_M=(V_1)_MV_M=V_M$ is t-closed. If $M\in \mathrm{MSupp}_{V_1}(V_1W_1/V_1)$, then $(V_1)_M=V_M$, so that $(V_1)_M(W_1)_M=V_M(W_1)_M$. It follows that $V_1W_1\subset VW_1$ is t-closed. For the same reason, $V_1W_1\subset V_1W$ is t-closed. 

Now, $V_1W_1\subset VW$ is also t-closed: assume that the contrary holds. Let $U'$ be the t-closure of $V_1W_1$ in $VW$ and $M\in\mathrm{MSupp}(U'/V_1W_1)$. Lemma \ref{9.12} implies that $M\not\in\mathrm{MSupp}(VW_1/V_1W_1)\cup\mathrm{MSupp}(V_1W/V_1W_1)$ because $R\subset S$ is catenarian. Then, $(VW_1)_M=V_M(W_1)_M=(V_1W_1)_M$
 
\noindent $=(V_1)_M(W_1)_M=(V_1W)_M=(V_1)_MW_M$, which leads to $(V_1W_1)_M=(VW)_M$, a contradiction because $M\in\mathrm{MSupp}(U'/V_1W_1)$ and $U'\in[V_1W_1,VW]$.  
 
Then, $V_1W_1\subset VW$ is a $\Delta$-extension by Proposition \ref{9.10} and Lemma \ref{9.11}, so that $VW_1+V_1W=VW_1V_1W=VW$ (4). 

Now  (2) and (3),  using (1) and (4),  combine to yield that $VW=VW_1+V_1 W=V+V_1W_1+W+V_1W_1=V+V_1+W_1+W+V_1+ W_1=V+W$. 

To conclude, we obtained $V+W=VW$ for any $V,W\in[R,S]$ and $R\subset S$ is a $\Delta$-extension.
\end{proof} 
 
\begin{corollary} \label{9.151} An unbranched integral FCP extension $R\subset S$ is a $\Delta$-extension if and only if $R\subset S$ is pinched at ${}_S^tR,\ R\subseteq{}_S^tR$ is a $\Delta$-extension and ${}_S^tR\subseteq S$ is chained. In particular, if these conditions hold,  $R\subset S$ is catenarian.
\end{corollary} 

\begin{proof} 
 Since $R\subset S$ is integral unbranched, then $S$ is local. Let $N$ be the maximal ideal of $S$. Then $N=({}_S^tR:S)$ and is also the maximal ideal of ${}_S^tR$ because ${}_S^tR\subseteq S$ is t-closed \cite[Proposition 4.10]{DPP2}. Moreover, ${}_S^+R={}_S^tR$ because $R\subset S$ is spectrally bijective by Proposition \ref{3.5}, so that any $T\in[R,S]$ is local \cite[Lemma 3.29]{Pic 11} and $R\subseteq  {}_S^tR$ is subintegral. 

Assume that $R\subset S$ is a $\Delta$-extension, and then is  catenarian by Proposition \ref{1.012}. Then $R\subseteq{}_S^tR$ is a $\Delta$-extension by Corollary \ref{9.2} and ${}_S^tR\subseteq S$ is chained by Proposition \ref{9.10}. Moreover, $R\subset S$ is pinched at ${}_S^tR$ according to \cite[Theorem 4.13]{Pic 12}.  

Conversely, assume that $R\subset S$ is pinched at ${}_S^tR,\ R\subseteq{}_S^tR$ is a $\Delta$-extension and ${}_S^tR\subseteq S$ is chained. Let $\mathcal C:=\{R_i\}_{i=0}^n$ be a maximal chain of $[R,S]$. Since any $R_i\in[R,{}_S^tR]\cup[{}_S^tR,S]$ and $R_{i-1}\subset R_i$ is minimal for each $i\in\mathbb N_n$, there exists some $k\in\{0,\ldots,n\}$ such that $R_k={}_S^tR$. If $k=0$, then $R\subset S$ is t-closed and chained, then a $\Delta$-extension. If $k=n$, then $R\subset S$ is a $\Delta$-extension by assumption. Now, assume that $k\in\mathbb N_{n-1}$. It follows that $\{R_i\}_{i=0}^k$ is a maximal chain of $[R,{}_S^tR]$ and  $\{R_i\}_{i=k}^n$ is a maximal chain of $[{}_S^tR,S]$. But, $\ell(\mathcal C)=n=k+(n-k)=\ell[R,{}_S^tR]+\ell[{}_S^tR,S]$ because $R\subseteq {}_S^tR$ is a  $\Delta$-extension and then catenarian by Proposition \ref{1.012} and ${}_S^tR\subseteq S$ is  chained. Hence, all the maximal chains of $[R,S]$ have the same length. Then $R\subset S$ is catenarian and is a $\Delta$-extension by Theorem \ref{9.15}.
\end{proof} 

In case of a t-closed extension, we recover Proposition \ref{9.10}, but this Proposition was necessary to prove Theorem \ref{9.15} and Corollary  \ref{9.151}.

\begin{example} \label{9.152} Here is an example of an integral FCP $\Delta$-extension $k\subset S$ such that $k$ is a local ring but $S$ is not a local ring, so that the extension is branched.
  As a corollary we get that $k\subset k^3$ is a $\Delta$-extension.
 
Let $k$ be a field, $k\subset k^3$ the ring extension defined as the diagonal map, $\{e_1,e_2,e_3\}$ the canonical basis of $k^3$, whose elements are idempotents of the ring $k^3$ and $k\subset K$ a minimal field extension. We know that $k\subset k^3$ has FIP by \cite[Proposition 2.1]{Pic 9}. According to \cite[Proposition 4.15]{DPP3}, $[k,k^3]=\{k,R_1,R_2,R_3,k^3\}$, where $R_i:=ke_i+k(1-e_i)$, for $i=1,2,3$. Set $R:=Ke_1+k(1-e_1)$ and $S:=Ke_1+ke_2+ke_3=Rk^3$.
In particular, $N:=(k^3:S)=ke_2+ke_3$, is a maximal ideal of $k^3$ and $S$. It is easily seen that we have the following commutative diagram, where {\it d}, (resp. {\it i}) indicates a minimal decomposed (resp. inert) extension. ($R_2$ and $R_3$ do not appear to make the situation clearer):
$$\begin{matrix}
{} & {}                             & {} & {}              & R & {}             & {} \\
{} & {}                             & {} & i\nearrow & {} & d\searrow & {} \\
k & \overset{d}\rightarrow & R_1& {}       & {} & {}           & {} S \\
{} & {}                             & {} & d\searrow & {} & i\nearrow & {} \\
{} & {}                             & {} & {}              & k^3  & {}          & {} 
\end{matrix}$$
In particular, $k\subset S$ is seminormal because so are $k\subset k^3$ and $k^3\subset S$. Then $k={}_S^+k$. Since $k^3={}_S^tk$, obviously, $R\not\in[k,{}_S^tk]\cup[{}_S^tk,S]$. We claim that $[k,S]=\{k,R_1,R_2,R_3,k^3,R,S\}$. We first show that there does not exist some $L\in[k,S]\setminus\{R_i\}_{i=1}^3$ such that $k\subset L$ is minimal. Suppose that the contrary holds. We proved above that $k\subset L$ can be neither decomposed, nor ramified since $L\not\in[k,k^3]$. Then, $k\subset L$ is inert, and a minimal field extension
with some $x\in L\setminus k$. But $L\subset S$ implies that $x=(y,z,t)$, with $y\in K$ and $z,t\in k$. Let $P(X)\in k[X]$ be the (monic) minimal polynomial of $x$ over $k$. It follows that $P(X)$ is irreducible in $k[X]$ and satisfies $P(y)=P(z)=P(t)=0$, a contradiction since $x\not\in k$.
Hence, there does not exists such $L$. We have just also proved that the only minimal extensions of $R_1$ are $R_1\subset k^3$ and $R_1\subset R$. Indeed, $R_1\subset R$ is the only minimal inert subextension of $R_1\subset S$ starting from $R_1$ and  $R_1\subset k^3$ is the only minimal decomposed subextension of $R_1\subset S$ starting from $R_1$.
 
Let $i\neq 1$, for instance $i=2$ (the proof would be the same for $i=3$). Then, $R_2\subset k^3$ is the only minimal decomposed extension of $R_2$ since $k^3={}_S^tk$. There does not exist any $U\in[R_2,S]$ such that $R_2\subset U$ is minimal ramified for the same reason. Assume that there is some $U\in[R_2,S]$ such that $R_2\subset U$ is minimal inert. Moreover, $R_2$ has two maximal ideals $N_1:=ke_2$ and $N_2:=k(1-e_2)=k(e_1+e_3)$ with $(R_2:k^3)=N_1$. We claim that $(R_2:U)=N_2$. Indeed, by \cite[Proposition 7.1]{DPPS}, $(R_2:U)=N_1$ leads to a contradiction: $k^3\subset S$ is not minimal. Then, $N_2\in\mathrm{Max}(U)$. But $N_1S=(ke_2)(Ke_1+ke_2+ke_3)=ke_2=N_1$ shows that $N_1=(R_2:S)$ is also an ideal of $U$, a contradiction since $N_1+N_2=R_2$. It follows that $k\subset k^3$ and $k\subset R$ are the only extensions of $k$ of length 2. Since $k^3\subset S$ and $R\subset S$ are minimal, we get that 
$[k,S]=\{k,R_1,R_2,R_3,k^3,R,S\}=[k,k^3]\cup\{R,S\}$. It remains to show that $U+V=UV$ for any $U,V\in[k,S]$. The result is obvious if either $U,V\in[k,k^3]$ or $U\in[k,k^3]$ and $V=S$, or $U,V\in\{R,S\}$.  The two last cases are when either $U\in[k,k^3]$ and $V=R$, or $U=R$ and $V\in[k,k^3]$, which are also obvious. Then, $k\subset S$ is a $\Delta$-extension. In particular, $k\subset k^3$ is a $\Delta$-extension,  by Corollary \ref{9.2} because $k^3$ is a subextension of $S=Rk^3$. 
\end{example} 

According to Theorem \ref{9.15}, we can reduce our study to infra-integral FCP extensions. These extensions are catenarian by Proposition~\ref{cat}. We will see in the next subsection that for such extensions, the $\Delta$-property is linked to extensions of length 2. 
  
 \subsection{Properties linked to extensions of length 2}

 We begin with the following  lemmas. 
 
 \begin{lemma} \label{9.16} Let $R\subset S$ be an  FCP extension and  $T,U\in[R,S],\ T\neq U$ be such that $\ell[T\cap U,T]+\ell[T\cap U,U]=\ell[T\cap U,TU]$. If $R\subset T$ and $R\subset U$ are minimal,  so are $T\subset TU$ and $U\subset TU$. In particular,  $R=T\cap U$.
\end{lemma} 

\begin{proof} Since we have the chain $R\subseteq T\cap U\subseteq T$ with $R\subset T$  minimal, it follows that either $T\cap U=R$ or $T\cap U=T$. In the last case, we should have $R\subset T\subset U$, with $R\subset U$  minimal, a contradiction. Then, $R=T\cap U$ and  $\ell[R,T]+\ell[R,U]=2=\ell[R,TU]$, with $T,U\subset TU$ both minimal. 
\end{proof} 

\begin{lemma} \label{9.17} Let $R\subset S$ be an  FCP extension and  $T,U\in[R,S],\ T\neq U$ be such that $R\subset TU$ is infra-integral with $R\subset T$ and $R\subset U$  minimal. Then, $R\subset TU$ is a $\Delta$-extension if and only if $\ell[R,TU]=2$.
\end{lemma} 

\begin{proof} Assume first that $R\subset TU$ is a $\Delta$-extension. Then, $[R,TU]$ is modular by Proposition \ref{1.012} and $\ell[R,TU]=2$ 
 by Lemma \ref{1.1}.

Conversely, assume that $\ell[R,TU]=2$ and let $V,W\in [R,TU]$. We are aiming to show that $V+W=VW$. The result is obvious if either $V=W$, or $V\in\{R,TU\}$, or $W\in\{R,TU\}$. In the other cases, $R\subset V,W$ and $V,W\subset TU$ are minimal infra-integral extensions. In particular, $VW=TU$. Set $M:={\mathcal C}(R,V)$ and $N:={\mathcal C}(R,W)$. 

Assume first that $M\neq N$. Then, Corollary \ref{1.015} shows that $R\subset TU$ is  a $\Delta$-extension.

Assume now that $M=N$. In particular, $M=(R:TU)$. Using Proposition  \ref{9.02}, working with the extension $R/M\subset (TU)/M$ and setting $k:=R/M$, we get that $\dim_k(V/M)=\dim_k(W/M)=2$ (by minimality of the infra-integral extensions $R\subset V$ and $R\subset W$ and Theorem \ref{minimal}). This leads to  $\dim_k((V+W)/M)=3$.  Since $k\subset V/M$ and $k\subset W/M$ are still minimal infra-integral extensions, there exist $x\in (V/M)\setminus k$ and  $y\in (W/M)\setminus k$ such that $V/M=k[x]=k+kx$ and $W/M=k[y]=k+ky$. Moreover, we may choose $x$ (resp. $y$) generating the (or one of the two) maximal ideal(s) of $V/M$ (resp. $W/M)$). Since we still have $\ell[k,(TU)/M]=\ell[k,(VW)/M]=2$, it follows from \cite[Proposition 7.6]{DPPS} that $xy=0$ so that $(VW)/M=k[x,y]=k+kx+ky$, giving $\dim_k((VW)/M)=3$. To end, $(V+W)/M\subseteq (VW)/M$ leads to $V+W=VW$ and $R\subset TU$ is  a $\Delta$-extension.
\end{proof} 

\begin{theorem} \label{9.171} An infra-integral  FCP extension of length two  is a $\Delta$-extension.
\end{theorem} 

\begin{proof} Let $R\subset S$ be an infra-integral FCP extension of length two. Let $T,U\in[R,S],\ T\neq U$ be such that  $R\subset T$ and $R\subset U$ are minimal. Since $T\subset TU\subseteq S$, it follows that $TU=S$, giving $\ell[R,TU]=2$. Hence, $R\subset S$ is a $\Delta$-extension by Lemma \ref{9.17}.
\end{proof}

\begin{lemma} \label{9.172} Let $R\subset S$ be an infra-integral FCP extension. For any $T,U,V\in[R,S],\ U\neq V$ such that $U\subset T$ and $V\subset T$ are minimal, then $\ell[U\cap V,T]=2$.
\end{lemma} 

\begin{proof} Set $W:=U\cap V,\ M:=(U:T)$ and $N:=(V:T)$. We use several dual results of Proposition \ref{3.6}. We claim that $M\neq N$. Deny. Then, \cite[Proposition 5.7]{DPPS} asserts that $M\in\mathrm{Max}(W)$, so that $M=(W:U)$, a contradiction because $W\subset U$ is infra-integral and not t-closed. Then, $M\neq N$.

If either $M$ and $N$ are incomparable or not, at least one of $W\subset U$ and $W\subset V$ is minimal  infra-integral by \cite[Proposition 6.6]{DPPS}. It follows that $\ell[W,T]=2$ because $W\subset T$ is infra-integral, and then catenarian according to Proposition \ref{cat}.
\end{proof} 

\begin{theorem} \label{9.18} Let $R\subset S$ be an infra-integral FCP extension. The following conditions are equivalent:
\begin{enumerate}
 \item $R\subset S$ is a $\Delta$-extension,
 
 \item   $R \subset S$ is modular,
 
\item $\ell[T\cap U,T]+\ell[T\cap U,U]=\ell[T\cap U,TU]$ for any $T,U\in[R,S]$,
 
\item  $R \subset S$ is semi-modular,

\item $\ell[T,UV]=2$ for any $T,U,V\in[R,S]$ such that $T\subset U$ and $T\subset V$ are minimal with $U\neq V$.
 \end{enumerate} 
\end{theorem} 

\begin{proof} (1) $\Rightarrow$ (2) is Proposition \ref{1.012}.

(2) $\Leftrightarrow$ (3) by \cite[Theorem 4.15]{R}. 

(3) $\Rightarrow$ (1) Assume that $\ell[T\cap U,T]+\ell[T\cap U,U]=\ell[T\cap U,TU]$ for any $T,U\in[R,S]$. Let $V,W\in[R,S]$. We claim that $V+W=VW$. Choose two maximal chains $R:=V_0\subsetÊV_1\subset \ldots\subset V_i\subset \ldots\subset V_n:=V$ and $R:=W_0\subsetÊW_1\subset\ldots\subset W_j\subset\ldots\subset W_m:=W$ such that $V_i\subset V_{i+1}$ and $W_j\subset W_{j+1}$ are minimal infra-integral for each $i\in\{0,\ldots,n-1\}$ and each $j\in\{0,\ldots,m-1\}$. We are going to show by induction on $k\in\{2,\ldots,m+n\}$ that $V_{i-1}W_j\subset V_iW_j$ and $V_iW_{j-1}\subset V_iW_j$ are minimal and that $V_i+W_j=V_iW_j$ for each $i,j\geq 1$ such that $i+j\leq k,\ i\leq n,\ j\leq m$. 

Let $k=2$, so that $i=j=1$. Since $R\subset V_1$ and $R\subset W_1$ are minimal, so are $V_1=V_1W_0\subset V_1W_1$ and $W_1=V_0W_1\subset V_1W_1$  by Lemma \ref{9.16}. Moreover, $\ell[R,V_1W_1]=2$, so that $V_1+W_1=V_1W_1$ by Lemma \ref{9.17}. Then, the induction hypothesis holds for $k=2$.

Assume that the induction hypothesis holds for $k<m+n$, so that for any $i,j$ such that $i+j\leq k$, we have $V_{i-1}W_j,\ V_iW_{j-1}\subset V_iW_j$ minimal and $V_i+W_j=V_iW_j$.  Let $\alpha,\beta$ be such that $\alpha+\beta=k+1$. Then, $(\alpha-1)+\beta=\alpha+(\beta-1)=k$ and the induction hypothesis holds for $k$. Consider the following commutative diagram:
$$\begin{matrix}
V_{\alpha-1} & \to & V_{\alpha} & \to & V_{\alpha}W_{\beta-1} & {} & {} \\
        {}          & \searrow & {}      & \nearrow & {}             & \searrow & {} \\
 {} & {} & V_{\alpha-1} W_{\beta-1} & {} & {} & {} & V_{\alpha}W_{\beta} \\
       {}          & \nearrow &                   {} & \searrow & {} & \nearrow & {} \\
W_{\beta-1} & \to & W_{\beta} & \to & V_{\alpha-1} W_{\beta} &  {} & {} 
\end{matrix}$$
It follows that $V_{\alpha-1}W_{\beta -1}\subset V_{\alpha-1}W_{\beta},\ V_{\alpha}W_{\beta-1}$ are minimal and so are 
 $V_{\alpha-1}W_{\beta},\  V_{\alpha}W_{\beta-1}\subset V_{\alpha}W_{\beta}$ by Lemma  \ref{9.16}. Now, $\ell[V_{\alpha-1}W_{\beta -1}, V_{\alpha}W_{\beta}]=2$, so that $V_{\alpha-1}W_{\beta}+  V_{\alpha}W_{\beta-1}=(V_{\alpha-1}W_{\beta}) (V_{\alpha}W_{\beta-1})=V_{\alpha}W_{\beta}\ (*)$ by Lemma \ref{9.17}. But $V_{\alpha-1}W_{\beta}= V_{\alpha-1}+W_{\beta}$ and  $V_{\alpha}W_{\beta-1}= V_{\alpha}+W_{\beta-1}$ by the induction hypothesis for $k$. Then, $(*)$ yields $V_{\alpha}W_{\beta}=V_{\alpha-1}W_{\beta}+  V_{\alpha}W_{\beta-1}=V_{\alpha-1}+W_{\beta}+ V_{\alpha}+W_{\beta-1}=V_{\alpha}+W_{\beta}$. Then, the induction hypothesis holds for $k+1$. As it holds for any $k$, in particular, for $k=n+m$, we get $V_n+W_m=V+W=V_nW_m=VW$. Then, $R\subset S$ is a $\Delta$-extension.
 
(2) $\Rightarrow$ (4) See the properties at the beginning of 
 Subsection 2.3.

(4) $\Rightarrow$ (5) Assume that $T\subset U$ and $T\subset V$ are minimal with $U\neq V$, so that $T=U\cap V$. (4) implies that $U\subset UV$ and $V\subset UV$ are minimal by definition of a semi-modular extension. It follows that $\ell[T,UV]=2$, giving (5).

(5) $\Rightarrow$ (2) In order to prove that $R\subset S$ is modular, we are going to show that both  covering conditions are satisfied, that is, for each $T,U\in[R,S]$ such that $T\cap U\subset T$ is minimal, then $U\subset TU$ is minimal, and for each $T,U\in[R,S]$ such that  $U\subset TU$ is minimal, then $T\cap U\subset T$ is minimal (see \cite[Definition page 105 and Theorem 4.15]{R}). First, let $T,U\in[R,S]$ be such that $T\cap U\subset T$ is minimal. We prove that $U\subset TU$ is minimal by induction on $l:=\ell[T\cap U,U]$. If $l=1$, then $T\cap U\subset U$ is minimal and $U\subset TU$ is minimal by (5) since  $\ell[T\cap U,TU]=2$. Assume that the induction hypothesis holds for $l-1$ and let $U'\in[T\cap U,U]$ be such that $U'\subset U$ is minimal, so that $T\cap U=T\cap U'$ and $\ell[T\cap U',U']=l-1$. From the induction hypothesis, we deduce that $ U'\subset TU'$ is minimal. Using (5) for the minimal extensions $U'\subset TU'$ and $U'\subset U$, we get that $\ell[U',TU]=2$, so that $U\subset TU$ is minimal.   

In the same way, let $T,U\in[R,S]$ be such that $U\subset TU$ is minimal. We prove that $T\cap U\subset T$ is minimal by induction on $r:=\ell[T,TU]$. If $r=1$, the result holds by Lemma \ref{9.172} since $T,U\subset TU$ are minimal. Assume that the induction hypothesis holds for $r-1$ and let $T'\in[T,TU]$ be such that $T\subset T'$ is minimal, so that $TU=T'U$ and $\ell[T',T'U]=r-1$. From the induction hypothesis, we deduce that $T'\cap U\subset T'$ is minimal. Using again Lemma \ref{9.172} for the minimal extensions $T\subset T'$ and $T'\cap U\subset T'$, we get that $T\cap U\subset T$ is minimal. To conclude, $R\subset S$ is modular.
\end{proof} 

We may remark that $(1)\Leftrightarrow (2)$ in the above theorem   re-proves Theorem  \ref{9.171} by Proposition  \ref{1.2}.
 
\begin{proposition} \label{9.181} Let $k$ be a field and $n$ a positive integer, $n>1$. Then, $k\subset k^n$ is an FIP $\Delta$-extension if and only if $n\leq 3$.
 \end{proposition} 

\begin{proof}  According to \cite[Proposition 2.1]{Pic 9}, for any positive integer $n,\ k\subset k^n$ is an FIP extension. For $n=2$, $k\subset k^2$ is a minimal extension by \cite[Lemme 1.2]{FO} and then a $\Delta$-extension by Proposition \ref{1.013}, since arithmetic. Example \ref{9.152} asserts that the result is valid for $n=3$.

Let $n\geq 4$ and let $\mathcal{B}:=\{e_1,\ldots,e_n\}$ be the canonical basis of the $k$-algebra $k^n$. Set $x:=e_1+e_2$ and $y:=e_1+e_3$. Then, $xy=e_1\in k[x,y]$ and $xy\not\in k[x]+k[y]=k[e_1+e_2]+k[e_1+e_3]=k+k(e_1+e_2)+k(e_1+e_3)$. Using Proposition \ref{9.0}, we get that $k\subset k^n$ is not a $\Delta$-extension.
\end{proof}

\begin{corollary} \label{9.182} Let $R\subset S$ be a seminormal infra-integral FCP extension. The following conditions are equivalent:
 \begin{enumerate}
 \item $R\subset S$ is a $\Delta$-extension;
 
\item $|\mathrm{V}_S(MS)|\leq 3$ for any $M\in\mathrm{MSupp}(S/R)$;
 
 \item $\ell[R_M,S_M]\leq 2$ for any $M\in\mathrm{MSupp}(S/R)$.\end{enumerate}
 \end{corollary} 

\begin{proof} (1) $\Leftrightarrow$ (2) According to Proposition \ref{9.02}, $R\subset S$ is a $\Delta$-extension if and only if $R_M\subset S_M$ is a $\Delta$-extension for any $M\in\mathrm{MSupp}(S/R)$. Moreover, $|\mathrm{V}_S(MS)|\leq 3$ for any $M\in\mathrm{MSupp}(S/R)$ if and only if $|\mathrm{Max}(S_M)|\leq 3$ for any $M\in\mathrm{MSupp}(S/R)$. Indeed, $\mathrm{Max}(S_M)=\{NS_M\mid N\in\mathrm{Max}(S),\ M\subseteq N\}=\{NS_M\mid N\in\mathrm{Max}(S),\ M\in\mathrm{V}_S(MS)\}$. Therefore, we can   assume that $(R,M)$ is a local ring. Since $R\subset S$ is a seminormal FCP extension, we deduce from \cite[Proposition 4.9]{DPP2}  that $(R:S)=M$ is an intersection of finitely many maximal ideals of $S$. Moreover,  $R\subset S$ being infra-integral,  $S/M\cong(R/M)^n$, where $n:=|\mathrm{Max}(S)|$ by the Chinese Remainder Theorem. Moreover, $R\subset S$ is a $\Delta$-extension if and only if $R/M\subset S/M$ is a $\Delta$-extension by Proposition \ref{9.02}. Since $R/M$ is a field, $R/M\subset (R/M)^n$ is a $\Delta$-extension if and only if $n\leq 3$ by Proposition~\ref{9.181}. To conclude, $R\subset S$ is a $\Delta$-extension if and only if $|\mathrm{V}_S(MS)|\leq 3$ for any $M\in\mathrm{MSupp}(S/R)$. 
 
(2) $\Leftrightarrow$ (3) by \cite[Lemma 5.4]{DPP2} since 
  $\ell[R_M,S_M]=|\mathrm{Max}(S_M)|-1=|\mathrm{V}_S(MS)|-1$. 
\end{proof}

\subsection{Applications of extensions of length 2 to  $\Delta$-extensions}

Before giving a more striking characterization of FCP infra-integral $\Delta$-extensions, we need the following technical lemmas.

\begin{lemma} \label{9.20} Let $R\subset S$ be an infra-integral  FCP extension over the local ring $(R,M)$, such that $R\subset{}_S^+R$ is a $\Delta$-extension and ${}_S^+R\subset S$ is minimal decomposed. 
Let $T,U\in[R,S],\ T\not\in[R,{}_S^+R]$ be such that $T\subset U$ is subintegral. Then the following hold:
\begin{enumerate}

 \item $|\mathrm{Max}(S)|=|\mathrm{Max}(T)|=|\mathrm{Max}(U)|=2$.

\item Let $V\in[T,U]$ and set $W:=V\cap {}_S^+R={}_V^+R$. Then, $\ell[T,V]=\ell[ {}_T^+R,W]$.
 
\item For any $V_1,V_2\in[T,U]$ such that $T\subset V_1,V_2$ are minimal, then $V_1+V_2=V_1V_2$ and $V_1,V_2\subset V_1V_2$ are minimal (ramified).

\item $T\subset U$ is a $\Delta$-extension.
 \end{enumerate}
\end{lemma} 

\begin{proof} (1) Since $(R,M)$ is a local ring, so is ${}_S^+R$ and $|\mathrm{Max}(S)| = 2$ since ${}_S^+R\subset S$ is minimal decomposed. We also have $|\mathrm{Max}(T)|=2$, because $T\not\in[R,{}_S^+R]$ implies that $R\subset T$ is not subintegral, which gives that $|\mathrm{Max}(R)|<|\mathrm{Max}(T)|\leq|\mathrm{Max}(S)|=2$. It follows that $|\mathrm{Max}(U)|=|\mathrm{Max}(T)|=2$ because $T\subset U$ is subintegral.

(2) Let $V\in[T,U]$ and $W:=V\cap {}_S^+R={}_V^+R$. Since $T\subset U$ is subintegral, so is $T\subset V$, which implies that $|\mathrm{Max}(V)|=2$. But, ${}_T^+R\subseteq T$ and $W\subseteq V$ are seminormal infra-integral. Therefore,  ${}_T^+R\subseteq T$ and $W\subseteq V$ are minimal decomposed and $\ell[{}_T^+R,T]=\ell[ W,V]=1$. Now, $\ell[{}_T^+R,V]=\ell[{}_T^+R,T]+\ell[T,V]=\ell[{}_T^+R,W]+\ell[W,V]$ because $R\subset S$ is infra-integral, whence catenarian  by Proposition~\ref{cat}. So, we get $\ell[T,V]=\ell[{}_T^+R,W]$.

(3) Let $V_1,V_2\in[T,U]$ be such that $T\subset V_1,V_2  $ are minimal (ramified). As in (2), set $W_i:={}_S^+R\cap V_i={}_{V_i}^+R$ for $i=1,2$, which are local rings. Since $T\subset V_1V_2$ is subintegral, because $V_1V_2\in[T,U]$, there is no harm to assume that $U=V_1V_2$. As $(T:V_1),(T:V_2)\in\mathrm{Max}(T)$, we consider two cases.

(a)  $(T:V_1)\neq(T:V_2)$. 

Propositions \ref{3.6} (1) and \ref{9.1} and Lemma \ref{9.17} say that $V_1+V_2=V_1V_2$. Moreover, Proposition \ref{3.6} shows that $ V_1,V_2  \subset V_1V_2  $ are minimal (ramified).

(b)  $(T:V_1)=(T:V_2)$. 

Set $N:=(T:V_1)=(T:V_2)$. For each $i\in\{1,2\}$, there is a unique $N_i\in\mathrm{Max}(V_i)$ lying above $N$ since $\mathrm{Max}(U)\to\mathrm{Max}(T)$ is bijective. Then, $M_i:=N_i\cap W_i$ is the maximal ideal of $W_i$. From (2), we deduce that ${}_T^+R\subset W_i$ is minimal ramified with conductor $M':=N\cap{}_T^+R$, the maximal ideal of ${}_T^+R$. In particular, $W_1,W_2\in[{}_T^+R,{}_S^+R]$, where ${}_T^+R\subseteq{}_S^+R$ is a $\Delta$-extension, because so is $R\subseteq{}_S^+R$. Then, $\ell[{}_T^+R,W_1W_2]=2$ by Lemma \ref{9.17}. An application of Proposition \ref{3.6} shows that $M_1M_2\subseteq M'\ (*)$. Since $|\mathrm{Max}(T)|=|\mathrm{Max}(V_i)|=2$, let $N'$ (resp. $N'_i$) be the other maximal ideal of $T$ (resp. of $V_i$). Then, $M'=NN'$ because $ {}_T^+R\subset T$ is minimal decomposed. For the same reason, we have  $M_i=N_iN'_i$ by Theorem \ref{minimal}. Then $(*)$ yields $N_1N'_1N_2N'_2=M_1M_2\subseteq M'=NN'\ (**)$. Since $N'\not\in\mathrm{MSupp}(V_i/T)$, we have $T_{N'}=(V_i)_{N'}=N_{N'}=(N_i)_{N'}$. Moreover, $N'_N=T_N$ and $(N'_i)_N=(V_i)_N$. Localizing  $(**)$ at the two maximal ideals $N,N'$ of $T$, we get $(N_1N'_1N_2N'_2)_N=(N_1N_2)_N\subseteq N_N$ and $(N_1N_2)_{N'}=T_{N'}=N_{N'}$. It follows that $N_1N_2\subseteq N$, so that $\ell[T,V_1V_2]=2$ by Proposition \ref{3.6} and $V_1+V_2=V_1V_2$ since $T\subseteq V_1V_2$ is a $\Delta$-extension by Lemma \ref{9.17}. In particular, $V_1\subset V_1V_2$ and $V_2\subset V_1V_2$ are minimal (ramified) extensions.

(4) Let $V,V'\in[T,U]$. We want to show that $V+V'=VV'$. Since $R\subset S$ has FCP, there exists a maximal chain $V_0:=T\subset  V_1\subset \ldots\subset V_i\subset\cdots\subset V_n:=V$ and a maximal chain $V'_0:=T\subset  V'_1\subset \cdots\subset V'_j\subset\cdots\subset V'_m:=V'$  such that $V_i\subset V_{i+1}$ and $V'_j\subset V'_{j+1}$ are minimal extensions,  for each $i=0,\ldots,n-1,\ j=0,\ldots,m-1$. We mimic the proof of Theorem \ref{9.18} although the assumptions are not the same.

We are going to show by induction on $k\in\{2,\ldots,m+n\}$ that $V_{i-1}V'_j\subset V_iV'_j$ and $V_iV'_{j-1}\subset V_iV'_j$ are minimal and $V_i+V'_j=V_iV'_j$ for each $i,j\geq 1$ such that 
$i+j\leq k,\ i\leq n,\ j\leq m$. 

If $k=2$, then $i=j=1$. Since $T\subset V_1$ and $T\subset V'_1$ are minimal, $V_1+V'_1=V_1V'_1$ by (3), and  the induction hypothesis holds for $k=2$.

Assume that the induction hypothesis holds for $k$. Hence,  for any $i,j$ such that $i+j\leq k$, we have that $V_{i-1}V'_j,\ V_iV'_{j-1}\subset V_iV'_j$ are minimal and $V_i+V'_j=V_iV'_j$.  Let $\alpha,\beta$ be such that $\alpha+\beta=k+1$ and consider the following commutative diagram:
$$\begin{matrix}
V_{\alpha-1} & \to & V_{\alpha} & \to & V_{\alpha}V'_{\beta-1} & {} & {} \\
        {}          & \searrow & {}      & \nearrow & {}             & \searrow & {} \\
 {} & {} & V_{\alpha-1} V'_{\beta-1} & {} & {} & {} & V_{\alpha}V'_{\beta} \\
       {}          & \nearrow &                   {} & \searrow & {} & \nearrow & {} \\
V'_{\beta-1} & \to & V'_{\beta} & \to & V_{\alpha-1} V'_{\beta} &  {} & {} 
\end{matrix}$$
Then, $(\alpha-1)+\beta=\alpha+(\beta-1)=k$. Since the induction hypothesis holds for $k$, it follows that $V_{\alpha-1}V'_{\beta -1}\subset V_{\alpha-1}V'_{\beta},\ V_{\alpha}V'_{\beta-1}$ are minimal. Moreover, $V_{\alpha-1}V'_{\beta -1}, V_{\alpha}V'_{\beta}\in[T,U]$, implies that $V_{\alpha-1}V'_{\beta -1}\subset V_{\alpha}V'_{\beta}$ is subintegral. Applying (3) to this extension, we get that $V_{\alpha-1}V'_{\beta}+  V_{\alpha}V'_{\beta-1}=(V_{\alpha-1}V'_{\beta}) (V_{\alpha}V'_{\beta-1})=V_{\alpha}V'_{\beta}\ (*)$. But $V_{\alpha-1}V'_{\beta}= V_{\alpha-1}+V'_{\beta}$ and  $V_{\alpha}V'_{\beta-1}= V_{\alpha}+V'_{\beta-1}$ by the induction hypothesis for $k$. Then, $(*)$ yields $V_{\alpha}V'_{\beta}=V_{\alpha-1}V'_{\beta}+  V_{\alpha}V'_{\beta-1}=V_{\alpha-1}+V'_{\beta}+ V_{\alpha}+V'_{\beta-1}=V_{\alpha}+V'_{\beta}$. Moreover, $V_{\alpha-1}V'_{\beta },\  V_{\alpha}V'_{\beta-1} \subset V_{\alpha}V'_{\beta}$ are minimal by (3). Then, the induction hypothesis holds for $k+1$. As it holds for any $k$, in particular, for $k=n+m$, we get $V_n+V'_m=V+V'=V_nV'_m=VV'$. Then, $T\subset U$ is a $\Delta$-extension.
\end{proof} 

\begin{corollary} \label{9.201} Let $R\subset S$ be an infra-integral  FCP extension over the local ring $(R,M)$, such that $R\subset {}_S^+R$ is a $\Delta$-extension and ${}_S^+R\subset S$ is minimal decomposed. Let $T,U\in[R,S]$ be such that  $T\subset U$ is subintegral. Then  $T\subset U$ is a $\Delta$-extension.
 \end{corollary} 

\begin{proof} According to Lemma \ref{9.20}, it is enough to verify only the case where $T\in [R,{}_S^+R]$, which is obvious since $U\in[R,{}_S^+R]$ and $R\subset {}_S^+R$ is a $\Delta$-extension (use Corollary \ref{9.2}).
\end{proof}

\begin{lemma} \label{9.21} Let $R\subset S$ be an infra-integral  FCP extension such that $(R,M)$ is a local ring, $R\subset {}_S^+R$ and ${}_S^+R\subset S$ are  $\Delta$-extensions. Let $T,U\in[R,S],\ T\not\in([R,{}_S^+R]\cup[{}_S^+R,S])$ be such that  $T\subset U$ is subintegral.  Then, $T\subset U$ is a $\Delta$-extension.
  \end{lemma} 

\begin{proof} Since ${}_S^+R\subset S$ is a $\Delta$-extension, 
$|\mathrm{Max}(U)|\leq|\mathrm{Max}(S)|\leq 3$ by Corollary \ref{9.182}. Since ${}_S^+R\neq S$, we have $|\mathrm{Max}(S)|\geq 2$ because $\ell[{}_S^+R,S]\geq 1$, and there is at least one minimal decomposed extension ${}_S^+R\subset V$ with $V\in[{}_S^+R,S]$. Then, use Theorem \ref{minimal}. We are done if $|\mathrm{Max}(S)|= 2$ by Lemma \ref{9.20}, since in this case ${}_S^+R\subset S$ is minimal decomposed. So, assume that $|\mathrm{Max}(S)|=3$. If $|\mathrm{Max}(U)|=2$, we may apply Lemma \ref{9.20} to the extension $R\subset U$ because ${}_U^+R\subseteq{}_S^+R$ and ${}_U^+R\subset U$ is minimal decomposed, in which case we are done. Assume now that $|\mathrm{Max}(U)|=|\mathrm{Max}(T)|=3$ since $T\subset U$ is subintegral. Then, $\ell[{}_T^+R,T]=|\mathrm{Max}(T)|-1=2$ by \cite[Lemma 5.4]{DPP2} since ${}_T^+R\subset T$ is seminormal infra-integral. It follows that there exists $T'\in[{}_T^+R,T]$ such that ${}_T^+R\subset T'$ and $T'\subset T$ are minimal decomposed. In particular, $|\mathrm{Max}(T')|=2$. Set $U':=T'{}_U^+R\in[{}_U^+R,U],\ n:=\ell[{}_T^+R,{}_U^+R]$, $R_0:={}_T^+R$ and $T_0=T'$. We are going to prove that there is some $m\leq n$ which will be exhibited below and by induction on $i\in\mathbb{N}_m$, 
 such that for each $i\in\mathbb{N}_m$, there exist $R_i\in[{}_T^+R,{}_U^+R]$ and $T_i\in[T',U']$ where $T_i=T'R_i$, $\ell[R_{i-1}, R_i]\in\{1,2\}$, $R_i\subset T_i$ is minimal decomposed and $T'\subset T_i$ is subintegral.
  Since $R\subset S$ has FCP, there exists $R'_1\in[{}_T^+R,{}_U^+R]$ such that $R_0\subset R'_1$ is minimal ramified. Consider $T'R'_1\in[T',U']$. Using Proposition \ref{3.6}, two cases occur according to the behavior of the maximal ideals of $T'$ and $R'_1$:

Case (a): $T'\subset T'R'_1$ is minimal ramified and $R'_1\subset T'R'_1$ is minimal decomposed. Then, we set $T_1:=T'R'_1\in[T',U']$ and $R_1:=R'_1$. It follows that $R_0\subset R_1$ is minimal ramified, $R_1\subset T_1$ is minimal decomposed and $T'\subset T_1$ is subintegral.

Case (b): We use \cite[Proposition 7.6]{DPPS}. There exists $T'_1\in[T', T'R'_1]$ such that $T'\subset T'_1$ and $T'_1\subset  T'R_1$ are minimal ramified and $R''_1\in[R'_1, T'R'_1]$ such that $R''_1\subset T'R'_1$ is minimal decomposed and $R'_1\subset R''_1$ is minimal ramified. We have the following commutative diagram where the horizontal maps are ramified and the vertical maps are decomposed:
$$\begin{matrix}
      T'=T_0       & \to  & T'_1 & \to &   T'R'_1  \\
\uparrow         &  {}   &    {}   &  {}  & \uparrow \\
{}_T^+R=R_0 & \to   & R'_1 & \to  & R''_1 
\end{matrix}$$
Set $T_1:=T'R'_1\in[T',U']$ and $R_1:=R''_1$. It follows that $\ell[R_0, R_1]=2$, $R_1\subset T_1$ is minimal decomposed,  $T'\subset T_1$ is subintegral and $T_1=T'R_1$. The induction hypothesis holds for $i=1$. 

Assume that the induction hypothesis holds for $i\in\mathbb{N}_{n-1}$, so that there exist $R_i\in[{}_T^+R,{}_U^+R[$ and $T_i\in[T',U']$ such that $\ell[R_{i-1}, R_i]\in\{1,2\},\ R_i\subset T_i$ is minimal decomposed,  $T'\subset T_i$ is subintegral and $T_i=T'R_i$. We mimic the proof made  for $i=1$, and use again Proposition \ref{3.6} and \cite[Proposition 7.6]{DPPS}. Since $R_i\in[{}_T^+R,{}_U^+R[$, there exists $R'_{i+1}\in[{}_T^+R,{}_U^+R]$ such that $R_i\subset R'_{i+1}$ is minimal ramified, and  two cases occur according to the behavior of the maximal ideals of $T_i$ and $R'_{i+1}$ (similar to the case $i=1$). 

In case (a): $T_i\subset T_iR'_{i+1}$ is minimal ramified and $R'_{i+1}\subset T_iR'_{i+1}$ is minimal decomposed. Then, we set $R_{i+1}:=R'_{i+1}\in[{}_T^+R,{}_U^+R[$ and $T_{i+1}:=T_iR_{i+1}\in[T',U']$. It follows that $R_i\subset R_{i+1}$ is minimal ramified, $R_{i+1}\subset T_{i+1}$ is minimal decomposed, $T'\subset T_{i+1}$ is subintegral and $T_{i+1}=T_iR_{i+1}=T'R_iR_{i+1}=T'R_{i+1}$. Then the induction hypothesis holds for $i+1$. 

In case (b): There exists $T'_i\in[T_i, T_iR'_{i+1}]\subseteq [T',U']$ such that  $T_i\subset T'_i$ and $T'_i\subset  T_iR'_{i+1}$ are minimal ramified and $R''_{i+1}\in [R_i, T_iR'_{i+1}]$ such that $R''_{i+1}\subset T_iR'_{i+1}$ is minimal decomposed and $R'_{i+1}\subset R''_{i+1}$ is minimal ramified. We do not draw a new diagram, it is enough to replace $0$ with $i$ and $1$ with $i+1$ in the above diagram. Then we set $T_{i+1}:=T_iR'_{i+1}$ and $R_{i+1}:=R''_{i+1}$. It follows that $\ell[R_i, R_{i+1}]=2,\ R_{i+1}\subset T_{i+1}$ is minimal decomposed, $T'\subset T_{i+1}$ is subintegral and $T_{i+1}=T_iR'_{i+1}=T_iR'_{i+1}R''_{i+1}=T_iR''_{i+1}=T'R_iR''_{i+1}=T'R''_{i+1}=T'R_{i+1}$ because $R'_{i+1}\subset R_{i+1}\subset T_iR'_{i+1}$. The induction hypothesis holds for $i+1$. 

Once $R_{i+1}={}_U^+R$ is gotten, we set $m:=i+1$, in which case $T_m=T'{}_U^+R=U'\in[{}_U^+R,U]$, with ${}_U^+R\subset U'$ minimal decomposed and $T'\subset U'$ subintegral. We can apply Lemma \ref{9.20} to the extension $R\subset U'$, which gives that $T'\subset U'$ is a $\Delta$-extension. We may remark that when we get that  $R_i\subset {}_U^+R$ is minimal (ramified), we are necessarily in case (a) since $R'_{i+1}={}_U^+R$.

Set $\mathrm{Max}(T'):=\{M,M'\}$, so that $(T'_M,M_M)$ is a local ring. We are first going to show that $T_M\subseteq U_M$ and $T_{M'}\subseteq U_{M'}$  are $\Delta$-extensions. 

Observe that the context is as follows:  $T'_M\subset U_M$ is an infra-integral  FCP extension where the local ring $(T'_M,M_M)$ verifies
 $U'_M={}_{U_M}^+T'_M,\ T'\subset U'$ is a $\Delta$-extension and so is $T'_M\subseteq U'_M$ by Proposition \ref{9.02}. Either $T'_M\subseteq T_M$ is  minimal decomposed or  $T'_M= T_M$. In this last case, $U_M=U'_M$, so that  $T_M\subseteq U_M$ is a $\Delta$-extension. If $T'_M\subseteq T_M$ is  minimal decomposed, the assumptions of Lemma \ref{9.20} are satisfied for the extension $T'_M\subseteq U_M$. As $T_M\subseteq U_M$ is a subintegral subextension of $T'_M\subseteq U_M$, we get that $T_M\subseteq U_M$ is a $\Delta$-extension. In the same way, $T_{M'}\subseteq U_{M'}$ is a $\Delta$-extension.
 
We intend to show that $T\subseteq U$ is a $\Delta$-extension  by applying twice Proposition \ref{9.02}. Indeed, $|\mathrm{Max}(T)|=3$.   Moreover, $(T':T)=M$ since $T'_M\subseteq T_M$ is  minimal decomposed. Let $N,P\in\mathrm{Max}(T)$ be lying over $M$ and $N'\in\mathrm{Max}(T)$ be lying over $M'$. Then, $T_{M'}=T_{N'},\ U_{M'}=U_{N'}$, and $T_N,T_P$ (resp. $U_N,U_P$) are localizations of $T_M$ (resp. $U_M$), so that  $T_N\subseteq U_N,\ T_P\subseteq U_P$ and $T_{N'}\subseteq U_{N'}$ are $\Delta$-extensions, and so is $T\subseteq U$.
\end{proof} 

\begin{corollary} \label{9.211} Let $R\subset S$ be an infra-integral  FCP extension over the local ring $(R,M)$, such that $R\subset{}_S^+R$ and ${}_S^+R\subset S$ are $\Delta$-extensions. Then a subintegral subextension $T\subset U$ of $ [R,S]$ is a $\Delta$-extension.
  \end{corollary} 

\begin{proof} According to Lemma \ref{9.21}, it is enough to assume that $T\in [R,{}_S^+R]\cup[{}_S^+R,S]$. If $T\in[R,{}_S^+R]$, so is $U$, then $T\subset U$ is a $\Delta$-extension because $R\subset {}_S^+R$ is a $\Delta$-extension. If $T\in[{}_S^+R,S]$, then $T\subset U\subseteq S$ implies that $U\in[{}_S^+R,S]$, a contradiction with $T\subset U$ is subintegral. Then, in any case, $T\subset U$ is a $\Delta$-extension
\end{proof} 

\begin{proposition} \label{9.23} An infra-integral FCP extension $R\subset S$ is a $\Delta$-extension if and only if the following statements  hold:
\begin{enumerate}

\item $R\subset {}_S^+R$ and ${}_S^+R\subset S$ are $\Delta$-extensions.
  
\item For each $T,U,V\in[R,S]$ such that $T\subset U$ is minimal ramified and $T\subset V$ is minimal decomposed, $T\subset  UV$ is a $B_2$-extension (or, equivalently, a $\Delta$-extension, or, equivalently, $\ell[T,UV]=~2$).
 \end{enumerate}
  \end{proposition} 

\begin{proof} One implication is obvious in the light of Corollary \ref{9.2} and Lemma \ref{9.17}. Moreover, in (2), the three conditions are equivalent. Indeed, if $T\subset U$ is minimal ramified and $T\subset V$ is minimal decomposed, then $T\subset UV$ is infra-integral by Proposition \ref{3.5}. Then, $T\subset UV$ is a $\Delta$-extension if and only if $\ell[T,UV]=2$ by Lemma \ref{9.17}. At last,\cite[Theorem 6.1 (5)]{Pic 6} gives that $\ell[T,UV]=2$ if and only if $|[T,UV]|=4$.

Conversely, assume that (1) and (2) hold. Since $R\subset S$ is infra-integral, there is no minimal inert subextension of $R\subset S$, always by Proposition \ref{3.5}. We can assume that $(R,M)$ is a local ring, and so is ${}_S^+R$. It follows that $|\mathrm{Max}(S)|\leq 3$ by Corollary \ref{9.182}. According to Theorem \ref{9.18}, it is enough to verify that $\ell[T,UV]=2$ for any $T,U,V\in[R,S]$ such that $T\subset U$ and $T\subset V$ are minimal with $U\neq V$. If they are minimal of different types, that is one is minimal ramified and the other minimal decomposed, $\ell[T,UV]=2$ by (2). Assume that $T\subset U$ and $T\subset V$ are both minimal  ramified, then $T\subset UV$ is subintegral. If $T\in[R,{}_S^+R]$, so are $U$ and $V$, then $T\subset UV$ is a $\Delta$-extension by (1), so that $\ell[T,UV]=2$ by Theorem \ref{9.18}. Moreover, we cannot have $T\in[{}_S^+R,S]$ since $T\subset U$ is minimal ramified. If $T\not\in[R,{}_S^+R]\cup[{}_S^+R,S]$, then $T\subset UV$ is a $\Delta$-extension by Lemma \ref{9.21}, so that $\ell[T,UV]=2$ by Theorem \ref{9.18}. At last, assume that $T\subset U$ and $T\subset V$ are both minimal decomposed. By Proposition \ref{3.6}, $\ell[T,UV]\leq 3$. But \cite[Proposition 7.6 (b)]{DPPS} says that when $\ell[T,UV]=3$, there exists $U'\in[U,UV]$ such that $U\subset U'$ and $U'\subset UV$ are both minimal decomposed. It follows that $|\mathrm{Max}(T)|+3=|\mathrm{Max}(UV)|\leq|\mathrm{Max}(S)|\leq 3$, a contradiction. Then, $\ell[T,UV]=2$. This equality holding in any case, $R\subset S$ is a $\Delta$-extension by Theorem \ref{9.18}.
\end{proof} 
\begin{remark} \label{9.230} The condition (2) of Proposition \ref{9.23} is necessary in order to have a $\Delta$-extension as it is shown in the following example. We use \cite[Example 5 of 16.4]{Ri} and its results.  Let $S$ be the ring of integers of $\mathbb Q(\sqrt 7,i)$. Then $S$ is a $\mathbb Z$-module of basis $\mathcal B:=\{1,\sqrt 7,t,u\}$, where $t:=(\sqrt 7+i)/2$ and $u:=(1+i\sqrt 7)/2$. It is shown that $2S=Q_1^2Q_2^2$, where $Q_1,Q_2$ are the maximal ideals of $S$ lying above $2\mathbb Z$. Set $R:=\mathbb Z+2S$ and $W:={}_S^+R$.  Then, $2S=(R:S)$ is a maximal ideal of $R$. It is easy to see that $W=\mathbb Z+Q_1Q_2$, so that $W\subset S$ is minimal decomposed, and then a $\Delta$-extension. Hence, the second  part of condition (1) 
of Proposition \ref{9.23} is satisfied. Set $U_1:=\mathbb Z+Q_1^2Q_2$ and $V_1:=\mathbb Z+Q_1^2$. A short calculation shows that $R\subset U_1$ and $V_1\subset S$ are both minimal ramified, while $U_1\subset V_1$ is minimal decomposed. Then, $\ell[R,S]=3$, which leads to $\ell[R,W]= 2$. It follows that $R\subset W$ is a $\Delta$-extension by Theorem \ref{9.171} and the first  part of condition (1)
of Proposition \ref{9.23} is satisfied. We show that condition (2) of Proposition \ref{9.23} is not satisfied.  

Set $T:=R[x]$, where $x:=1+\sqrt 7$. Since $x^2=8+2\sqrt 7\in 2S$ and $2xS\subseteq 2S$, we get that $R\subset T$ is minimal ramified by Theorem \ref{minimal}. Set $T':=R[u]$. Since $u^2-u=-2\in 2S$ and $2uS\subseteq 2S$, we get that $R\subset T'$ is minimal decomposed by Theorem \ref{minimal}. Now, $TT'=R[x,u]=R[\sqrt 7,(1+i\sqrt 7)/2]$, with $\sqrt 7(1+i\sqrt 7)/2=t+3i\in TT'$. But, $i=2t-\sqrt 7\in T\subset TT'$ implies that $t\in TT'$. To conclude $TT'=S$ and $\ell[R,TT']>2$ shows that $R\subset S$ is not a $\Delta$-extension. We can check that $T+T'$ is not a ring, because it does not contain $\sqrt 7u$.

We will see in Section 5 an example where conditions of Proposition \ref{9.23} hold in the context of number field orders.   
 \end{remark} 

\begin{lemma} \label{9.231} Let $R\subset S$ be an    integral FCP extension such that $\ell[T,UV]=2$ for each $T,U,V\in[R, S]$ such that $T\subset U$  and $T\subset V$ are minimal of different types. Then, for any maximal chain $\mathcal C'$ of $[R,S]$, there exists a maximal chain $\mathcal C$ of $[R,S]$ containing ${}_S^tR$ such that $\ell(\mathcal C)=\ell(\mathcal C')$.
\end{lemma} 

\begin{proof} Let $\mathcal C'$ be a maximal chain  of $[R,S]$. According to \cite[Proposition 4.11]{Pic 12}, there exists a maximal chain $\mathcal C$ of $[R,S]$ containing ${}_S^tR$ such that $\ell(\mathcal C)\geq \ell(\mathcal C')$. Assume that $\ell(\mathcal C)\neq \ell(\mathcal C')$. The same reference shows that there exist $A,B,C\in\mathcal C'$ such that $A\subset B$ is inert, $B\subset C$ is minimal non-inert, and $(A:B)=(B:C)$. Set $W:={}_C^tA\neq A,C$ because $A\subset C$ is neither infra-integral nor t-closed, and $M:=(A:B)=(B:C)\in\mathrm{Max}(A)$. In particular, $M=(A:C)$.  Let $W_1\in[A,W]$ be such that $A\subset W_1$ is minimal non-inert. In particular, $M=(A:W_1)$. Then, Proposition \ref{3.6} asserts that $A\subset BW_1$ is not catenarian, a contradiction with $\ell[A,BW_1]=2$ by assumption and \cite[Proposition 3.4]{Pic 12}. To conclude, $\ell(\mathcal C)= \ell(\mathcal C')$.
  \end{proof}

 \begin{theorem} \label{9.24} An   FCP  extension $R\subset S$ is a $\Delta$-extension if and only if the following conditions hold:
 
\begin{enumerate}
\item $R\subset {}_{\overline R}^+R$, ${}_{\overline R}^+R\subset{}_{\overline R}^tR$ and ${}_{\overline R}^tR\subseteq \overline R$ are $\Delta$-extensions (the last condition being equivalent to ${}_{ \overline R}^tR\subseteq \overline R$ is arithmetic).
 
\item For each $T,U,V\in[R, \overline R]$ such that $T\subset U$  and $T\subset V$ are minimal of different types, then $T\subset UV$ is a $B_2$-extension.
\end{enumerate}
\end{theorem} 

\begin{proof} Assume that $R\subset S$ is a $\Delta$-extension, and then catenarian, so that (2) holds by Lemma \ref{1.1} and Corollary \ref{1.015}. Moreover, (1) holds because of Proposition  \ref{9.10} and Corollary \ref{9.2}.
 
Conversely, assume that conditions (1) and (2) hold. By Proposition  \ref{9.23}, we get that $R\subseteq {}_{ \overline R}^tR$ is a $\Delta$-extension thanks to (1) and (2), and is catenarian since infra-integral.

We now observe that  (2) implies that  $R\subseteq \overline R$ is catenarian. Indeed, any maximal chain from $R$ to ${}_{ \overline R}^tR$ has the same length $\ell[R,{}_{ \overline R}^tR]=:k$, and any maximal chain from ${}_{ \overline R}^tR$ to $\overline R$ has the same length $\ell[{}_{ \overline R}^tR,\overline R]:=l$ because ${}_{ \overline R}^tR\subseteq \overline R$ is a $\Delta$-extension 
 by  Proposition \ref{1.012}, whence catenarian. Now, let $R_0:=R\subset\cdots\subset R_i\subset\cdots  R_n:=\overline R$ be a maximal chain $\mathcal C$ such that $R_i\subset R_{i+1}$ is a minimal extension for each $i=0,\ldots,n-1$. Set $m:=k+l$. If ${}_{ \overline R}^tR\in\mathcal C$, then ${}_{ \overline R}^tR=R_k$ and $\ell(\mathcal C)=n=m$ because $R\subseteq {}_{ \overline R}^tR$ and ${}_{ \overline R}^tR\subseteq \overline R$ are catenarian. Now, assume that ${}_{ \overline R}^tR\not\in\mathcal C$. According to Lemma \ref{9.231}, (2) implies that there exists a maximal chain $\mathcal C'$ of $[R,S]$ containing ${}_{ \overline R}^tR$ such that $\ell(\mathcal C)=\ell(\mathcal C')=m$. It follows that any maximal chain of $[R,\overline R]$ have length $m$, and $R\subseteq \overline R$ is catenarian.

Now, Theorem \ref{9.15} shows that $R\subseteq\overline R$ is a $\Delta$-extension, which implies at last that $R\subset S$ is a $\Delta$-extension by Theorem  \ref{9.7}.
\end{proof} 

 Adding some assumptions in the statement of Theorem \ref{9.24}, we get a simpler characterization of  $\Delta$-extensions.

\begin{corollary} \label{9.241} Let $R\subset S$ be a catenarian FCP extension. Then, $R\subset S$ is a $\Delta$-extension if and only if the following conditions hold:
 \begin{enumerate}
 \item $R\subset {}_{ \overline R}^+R$, ${}_{ \overline R}^+R\subset {}_{ \overline R}^tR$ and ${}_{ \overline R}^tR\subseteq \overline R$ are  $\Delta$-extensions  (the last condition being equivalent to ${}_{ \overline R}^tR\subseteq \overline R$ is arithmetic).
  
\item For each $T,U,V\in[R, \overline R]$ such that $T\subset U$  and $T\subset V$ are minimal non-inert of different types, then $T\subset UV$ is a $B_2$-extension.
 \end{enumerate}
\end{corollary} 
\begin{proof} One implication is Theorem \ref{9.24}. Conversely, in order to prove that $R\subset S$ is a $\Delta$-extension, it is enough to prove that (2) implies condition (2) of Theorem \ref{9.24}. So, let $T\subset U$  and $T\subset V$ be minimal extensions of different types, with one of them inert. Since $R\subset S$ is a catenarian extension, Proposition \ref{1.2} says that $\mathcal{C}(T,U)\neq\mathcal{C}(T,V)$, and
$[T,UV]=\{T,U,V,UV\}$. Then, $T\subset UV$ is a $B_2$-extension. Therefore, $R\subset S$ is a $\Delta$-extension.
\end{proof}

\begin{corollary} \label{9.25} A modular FCP extension $R\subset S$ is a $\Delta$-extension if and only if ${}_{\overline R}^tR\subseteq \overline R$ is arithmetic.
\end{corollary} 

\begin{proof} If $R\subset S$ is a $\Delta$-extension, 
 ${}_{\overline R}^tR\subseteq\overline R$ is arithmetic by Corollary \ref{9.241}.
 Conversely, assume that ${}_{\overline R}^tR\subseteq\overline R$ is arithmetic. Then, Corollary \ref{9.241} (1) holds since $R\subset {}_{\overline R}^+R$ is modular, and then a $\Delta$-extension by Theorem \ref{9.18}. Let $T,U,V\in[R,\overline R]$ be such that $T\subset U$ and $T\subset V$ are minimal non-inert of different types. Then, they are infra-integral, and so is $T\subset UV$, which is also modular. It follows that $\ell[T,UV]=2$ by Theorem \ref{9.18}, so that $[T,UV]=\{T,U,V,UV\}$ by \cite[Theorem 6.1 (5)]{Pic 6}. Therefore, $T\subset UV$ is a $B_2$-extension. To conclude, $R\subset S$ is a $\Delta$-extension  by Corollary \ref{9.241}.
  \end{proof} 
 
 \section{Examples}
 
 The  preceding section shows that we are lacking of a characterization of arbitrary subintegral $\Delta$-extensions, except  those of Theorem \ref{9.18} which says that an infra-integral FCP extension is a  $\Delta$-extension if and only if it is modular, this last condition being unsatisfactory.  
 In this section, we give examples of subintegral $\Delta$-extensions with various properties. We also characterize some special types of FCP extensions that are $\Delta$-extensions. 
 
 \begin{proposition}\label{11.01}  Let $R\subset S$ be an   FCP subintegral extension. Assume that $(R:S)_M\neq M_M$ for any $M\in\mathrm{MSupp}(S/R)$.   Then $R\subset S$ is a  $\Delta$-extension .
\end{proposition}

\begin{proof} According to Proposition \ref{9.02}, we may assume first that $(R,M)$ is a local ring such that $M\neq (R:S)$, and after, considering the factor ring $R/(R:S)$, so that $(R:S)= 0$, with $R$ which is a local Artinian ring by \cite[Theorem 4.2]{DPP2}, and not a field. First assume that $R/M$ is infinite. Then \cite[Proposition 5.15]{DPP2} asserts that $R\subset S$ is chained, and then a $\Delta$-extension by Proposition \ref{1.013}. Assume now that $R/M$ is finite, and so is $R$ because Artinian \cite[Lemma 5.11]{DPP2}. Set $R':=R(X),\ S':=S(X)$ and $M':=MR(X)$. Then $R'$ is a local Artinian ring with maximal ideal $M'$ which has the same index of nilpotency as $M$, so that $R'$ is not a field, $R'/M'$ is infinite and $(R':S')=0$. Then, the first part of the proof says that $R'\subset S'$ is a  $\Delta$-extension, and so is $R\subset S$ by Proposition  \ref{8.14}.
\end{proof} 

Since a $\Delta$-extension $R\subset S$ is modular (Proposition \ref{1.012}), a first approach consists in exhibiting conditions that imply the distributivity of a modular extension. For a distributive FCP extension $R\subset S$ and its Loewy series $\{S_i\}_{i=0}^n$ (see the definition before Corollary \ref{11.2}), \cite[Proposition 3.8]{Pic 11} says that  $S_i\subset S_{i+1}$ is Boolean for each  $i\in\{0,\ldots,n\}$. 

We begin with the characterization of Boolean $\Delta$-extensions. In  the light of \cite[Proposition 3.5]{Pic 10}, we first consider extensions $R\subset S$ over a local ring.

\begin{proposition}\label{11.1} Let $R\subset S$ be a Boolean FCP  extension, where $(R,M)$ is a local ring. Then $R\subset S$ is a $\Delta$-extension if and only if either $R\subset S$ is minimal, or an 
  infra-integral extension (in fact $B_2$).
\end{proposition}

 \begin{proof} Let $R\subset S$ be a Boolean FCP extension. By \cite[Theorem 3.30]{Pic 10},  one of the following conditions is satisfied:
 \begin{enumerate}
 \item $R\subset S$ is a minimal extension.

\item There exist $U,T\in[R,S]$ such that $R\subset T$ is minimal ramified, $R\subset U$ is minimal decomposed and $[R,S]=\{R,T,U,S\}$.

\item $R\subset S$ is a Boolean t-closed extension.
\end{enumerate}

In case (1), $R\subset S$ is a $\Delta$-extension by Proposition \ref{1.013}. 

In case (2), $S=TU$, $\ell[R,S]=2$ and $R\subset S$ is  a  $B_2$-extension, because $T\neq U$. In particular, $R\subset S$ is infra-integral since $S=TU={}_S^tR$. Then Theorem \ref{9.171} implies that $R\subset S$ is a $\Delta$-extension. To sum up, a Boolean infra-integral extension is always a $\Delta$-extension.  
 
In case (3), Proposition \ref{9.10} shows that $R\subset S$ is a $\Delta$-extension if and only if $R\subset S$ is chained. But a chain is Boolean if and only if it is minimal by \cite[Example 3.3 (1)]{Pic 10}. Then, we recover case (1).
\end{proof} 

In \cite{Pic 11}, we studied the Loewy series $\{S_i\}_{i=0}^n$  associated to an FCP ring extension $R\subseteq S$ defined as follows in \cite[Definition 3.1]{Pic 11}: the {\it socle} of the extension $R\subset S$ is $\mathcal S[R,S]:=\prod_{A\in \mathcal A}A$, the product of the atoms of $[R,S]$, and the {\it Loewy series} of the extension $R\subset S$ is the chain $\{S_i\}_{i=0}^n$ defined by induction: $S_0:=R,\ S_1:=\mathcal S[R,S]$ and for each $i\geq 0$ such that $S_i\neq S$, we set $S_{i+1}:=\mathcal S[S_i,S]$. Of course, since $R\subset S$ has FCP, there is some integer $n$ such that $S_{n-1}\neq S_n=S_{n+1}=S$. We said that $R\subset S$ is a {\it $\mathcal P$-extension} if $R\subset S$ is pinched at the chain $\{S_i\}_{i=1}^{n-1}$ \cite[Definition 3.19]{Pic 11}.

\begin{corollary} \label{11.2} A  distributive FCP $\mathcal P$-extension $R\subset S$ with Loewy series $\{S_i\}_{i=0}^n$ is a $\Delta$-extension if and only if $S_i\subset S_{i+1}$ is locally either minimal or an infra-integral $B_2$-extension for each $0\leq i\leq n-1$.
\end{corollary} 

\begin{proof} By \cite[Proposition 3.8]{Pic 11}, $S_i\subset S_{i+1}$ is Boolean for each $0\leq i\leq n-1$ because $R\subset S$ is distributive. Since $R\subset S$ is a $\mathcal P$-extension, $R\subset S$ is pinched at the chain $\{S_i\}_{i=1}^{n-1}$. Then, $R\subset S$ is a $\Delta$-extension if and only if $S_i\subset S_{i+1}$ is a $\Delta$-extension for each $0\leq i\leq n-1$ by Corollary \ref{9.03}. To end, Propositions \ref{9.02} and \ref{11.1} show that $S_i\subset S_{i+1}$ is  a $\Delta$-extension if and only if $S_i\subset S_{i+1}$ is locally a $\Delta$-extension if and only if $S_i\subset S_{i+1}$ is locally either minimal or an infra-integral $\mathcal B_2$-extension.
\end{proof}

\begin{proposition} \label{11.4} A locally unbranched distributive FCP  
extension $R\subset S$ is a $\Delta$-extension if and only if it  is arithmetic.
\end{proposition} 

\begin{proof} If $R\subset S$ is arithmetic, then $R\subset S$ is a  $\Delta$-extension
by Proposition  \ref{1.013}.

Conversely, assume that $R\subset S$ is a  $\Delta$-extension.  According to Proposition \ref{9.02}, there is no harm to assume that $(R,M)$ is a local ring. We begin to show that $R\subseteq \overline R$ is  chained. 

 We may assume $R\neq \overline R$. Since $R\subset S$ is distributive, so is $R\subset \overline R$. The assumption yields that  $\overline R$ is local, as any element of $[R,\overline R]$. Set $T:={}_{ \overline R}^+R= {}_{ \overline R}^tR$ because $R\subset S$ is unbranched (Proposition \ref{3.5}). Then, Corollary \ref{9.151} shows that $[R,\overline R]=[R,T]\cup[T,\overline R]$ with $T\subseteq \overline R$   chained. Now, assume that $R\neq T$.

If  $\{S_i\}_{i=0}^n$ is the  Loewy series  of $R\subset T$, then  $S_i\subset S_{i+1}$ is a Boolean extension for each $i\in\{0,\ldots,n-1\}$ according to \cite[Proposition 3.8]{Pic 11}. But $S_i\subset S_{i+1}$ is also a $\Delta$-extension, with $S_i$ a local ring for each $i\in\{0,\ldots,n-1\}$. Now  $S_i\subset S_{i+1}$ is necessarily minimal  ramified, because $R\subset T$ is subintegral (see the conditions of \cite[Theorem 3.30]{Pic 10} in the proof of Proposition  \ref{11.1}). An easy induction shows that $R\subset T$ is chained: since $S_0=R\subset S_1$ is  minimal, $R\subset T$ has only one atom (which is $S_1$), and so has  any subextension $S_i\subset T$ for each $i\in\{0,\ldots,n-1\}$, whence $R\subset T$ is chained. Now, $[R,S]=[R,T]\cup[T,\overline R]\cup[\overline R,S]$ according to Lemma \ref{1.12}, because  there does not exist some $V\in[R,S]\setminus([R,\overline R]\cup[\overline R,S])$ since $V\cap \overline R=\overline R$. But $\overline R\subseteq S$ is also chained, because $\overline R$ being  local, \cite[Theorem 6.10]{DPP2} implies that $\overline R\subset S$ is chained. To conclude, $R\subset S$ is chained. 
\end{proof} 

We saw in Theorem \ref{9.171} that an infra-integral extension of length 2 is a $\Delta$-extension. The next proposition characterizes  $\Delta$-extensions of length 2.

\begin{proposition} \label{11.41} A length 2 extension $R\subset S$ is a  $\Delta$-extension, except when the three following conditions are together verified: $|\mathrm{MSupp}(S/R)|$

\noindent $=1,\ R\subset S$ is t-closed and $|[R,S]|>3$.
\end{proposition} 

\begin{proof} We use the characterization of ring extensions of length 2 described in \cite[Theorem 6.1]{Pic 6} and keep only cases involving $\Delta$-extension: A ring extension $R\subset S$ is of length 2 if and only if one of the following conditions hold: 
\begin{enumerate}

\item $|\mathrm{Supp}(S/R)|=2,\ \mathrm{Supp}(S/R)\subseteq \mathrm{Max}(R)$ and $|[R,S]|=4$. In this case, for each $M\in\mathrm{Supp}(S/R),\ R_M\subset S_M$ is minimal, so that $R\subset S$ is arithmetic, and then a $\Delta$-extension by Proposition \ref{1.013}.

\item $|\mathrm{Supp}(S/R)|=2,\ \mathrm{Supp}(S/R)\not\subseteq \mathrm{Max}(R)$ and $|[R,S]|=3$. 

\item $R\subset S$ is a non-integral $M$-crucial extension and $|[R,S]|=3$. 

\item $R\subset S$ is an integral $M$-crucial extension such that ${}_S^tR\neq R,S$ and $|[R,S]|=3$. 

In cases (2), (3) and (4), $R\subset S$ is chained
 and then a  $\Delta$-extension by Proposition \ref{1.013}.

\item $R\subset S$ an infra-integral $M$-crucial extension such that ${}_S^+R\neq R,S$ and either  $|[R,S]|=3$ or $(R:S)= M$ with $|[R,S]|=4$. 

\item $R\subset S$ is a subintegral $M$-crucial extension of length 2
 with either $|[R,S]|=3$  or  $|[R,S]|=|R/M]+3$.

\item $R\subset S$ is a seminormal infra-integral $M$-crucial extension such that  $|[R,S]|=5$.

In  cases (5), (6) and (7), $R\subset S$ is infra-integral and then a  $\Delta$-extension by Theorem \ref{9.171}.

\item $R\subset S$ is a t-closed integral $M$-crucial extension of length 2, so that $M=(R:S)$, and $R/M\subset S/M$ is a field extension. 
According to Proposition \ref{9.10}, a t-closed extension $R\subset S$ is a $\Delta$-extension if and only if $R\subset S$ is arithmetic if and only if $R/M\subset S/M$ is chained for $M\in \mathrm{MSupp}(S/R)$. But $|\mathrm{MSupp}(S/R)|=1$, so that $R/M\subset S/M$ is chained for $M\in\mathrm{MSupp}(S/R)$ if and only if $R\subset S$ is chained if and only if $|[R,S]|=3$.
\end{enumerate} 

To sum up, a ring extension $R\subset S$ of length 2 is a $\Delta$-extension except when $R\subset S$ is a  t-closed integral extension such that $|\mathrm{MSupp}(S/R)|=1$ and $|[R,S]|>3$.
\end{proof} 

\begin{corollary} \label{11.42} A $B_2$-extension is a  $\Delta$-extension except if it is  t-closed.
\end{corollary} 

\begin{proof} Let $R\subset S$ be a $B_2$-extension, that is $\ell[R,S]=2$ and $|[R,S]|=4$. According to the characterization of length 2 extensions recalled in the proof of Proposition \ref{11.41}, $|[R,S]|=4$ happens in the following cases of this proof: (1) and then $R\subset S$ is a $\Delta$-extension; (5) and $R\subset S$ is a $\Delta$-extension. In the t-closed case, when $R\subset S$ is a $\Delta$-extension, $R\subset S$ is chained, so that $|[R,S]|=3\neq 4$ when $\ell[R,S]=2$.
\end{proof}
 
We now consider the $\Delta$-properties for pointwise minimal extensions. A ring extension $R\subset S$ is {\it pointwise minimal} if $R\subset R[t]$ is minimal for each $t\in S\setminus R$. We characterized these extensions in a joint work with Cahen in \cite{Pic 7}. The properties of pointwise minimal extensions $R\subset S$ allow to assume that $(R,M)$ is a local ring. In \cite[Theorems 3.2 and 5.4 and Proposition 3.5]{Pic 7}, we gave the different conditions for an extension $R\subset S$ to be pointwise minimal. Firstly, a pointwise minimal extension is either integral or integrally closed, and, in this last case, is minimal. As minimal extensions are $\Delta$-extensions, we need only to consider pointwise minimal integral extensions which are not minimal, that we describe as follows. Let $R\subset S$ be an integral extension such that $(R,M)$ is a local ring and which is not minimal. Then $R\subset S$ is a pointwise minimal extension if and only if $M=(R:S)$ and one of the following condition is satisfied: 

($\alpha$) $R/M=\mathbb Z/2\mathbb Z$ and $S/M\cong (R/M)^n$ for some integer $n\geq 3$. In this case, $R\subset S$ is seminormal infra-integral.
 
($\beta$) $M\in\mathrm{Max}(S),\ \mathrm{c}(R/M)=p$, a prime integer, and $R/M\subset S/M$ is a 
 purely inseparable
 extension where $x^p\in R$ for each $x\in S$. In this case, $R\subset S$ is t-closed.
 
($\gamma$) $R\subset S$ is subintegral with $\mathrm{Max}(S)=\{N\}$ such that $x^2\in M$ for each $x\in N$.
 
($\delta$) $\mathrm{Max}(S)=\{N\}$ is such that $x^2\in M$ for each $x\in N,\ R\subset R+N\subset S,\ \mathrm{c}(R/M)=p$, a prime integer and $x^p\in R$ for each $x\in S$. In this case, the spectral map of $R\subset S$ is bijective. 
 
\begin{proposition} \label{11.5} A pointwise minimal FCP extension $R\subset S$ over the local ring $(R,M)$ is a $\Delta$-extension if and only if one of the following conditions holds:
 \begin{enumerate} 
 \item $R\subset S$ is minimal.
 
\item $R\subset S$ is seminormal infra-integral with $|\mathrm{Max}(S)|=3$.
    
\item $R\subset S$ is subintegral with $N^2\subseteq M$, where 
  $\mathrm{Max}(S)=\{N\}$.
 \end{enumerate}
 \end{proposition} 
\begin{proof} Since we have recalled just before the characterization of pointwise minimal extensions $R\subset S$ , we examine, for each case, a necessary and sufficient condition in order that $R\subset S$ be a $\Delta$-extension. So, in the following, we assume that $R\subset S$ is a pointwise minimal extension.

If $R\subset S$ is integrally closed, $R\subset S$ is minimal, and then a $\Delta$-extension and we recover case (1).

If $R\subset S$ is not integrally closed, then $R\subset S$ is a integral extension. If $R\subset S$ is minimal,   then $R\subset S$ is  a $\Delta$-extension and we recover case (1). 

If $R\subset S$ is a integral extension which is not minimal,   then $R\subset S$ satisfies one of conditions $(\alpha)$--$(\delta)$ and 
 moreover, $M=(R:S)$.

 In case ($\alpha$), we have  $R/M=\mathbb Z/2\mathbb Z$ and $S/M\cong (R/M)^n$ for some integer $n\geq 3$. In this case, $R\subset S$ is seminormal infra-integral, and by Corollary \ref{9.182}, $R\subset S$ is a $\Delta$-extension if and only if  $|\mathrm{V}_S(MS)|\leq 3$ for any $M\in\mathrm{MSupp}(S/R)$, which is equivalent, since $(R,M)$ is a local ring, to $|\mathrm{Max}(S)|\leq 3$. But $|\mathrm{Max}(S)|=2$ implies that $R\subset S$ is minimal, a contradiction. To sum up, if $R\subset S$ is seminormal infra-integral non minimal, then $R\subset S$ is a $\Delta$-extension if and only if  $|\mathrm{Max}(S)|=3$ which gives case (2).

The condition $(\beta)$ is not involved for a t-closed non-minimal $\Delta$-extension $R\subset S$ where $(R,M)$ is a local ring since it is  chained by Proposition \ref{9.10}, and so is $R/M\subset S/M$. In fact, in this case, $R/M\subset S/M$ is chained if and only if $R/M\subset S/M$ is minimal: assume that the contrary holds. Then $R/M\subset S/M$ is a non-minimal chain satisfying condition $(\beta)$. Let $S'\in[R/M,S/M]$ be such that $R/M\subset S'$ is minimal, which implies that $S'\neq S/M$. Let $x\in(S/M)\setminus S'$, so that $x^p\in R/M$, giving $R/M\subset (R/M)[x]$ is minimal, with $S'\neq  (R/M)[x]$, a contradiction. It follows that $R\subset S$ needs to be a minimal extension and we recover case (1).

If case ($\gamma$) holds, $R\subset S$ is subintegral with $\mathrm{Max}(S)=\{N\}$ such that $x^2\in M$ for each $x\in N$. In particular, $R[x]=R+Rx$ for each $x\in N$. We claim that $R\subset S$ is a $\Delta$-extension if and only if $N^2\subseteq M$. But $N^2\subseteq M$ if and only if $xy\in M$ for each $x,y\in N$. According to Proposition \ref{9.0}, $R\subset S$ is a  $\Delta$-extension if and only if $R[s,t]=R[s]+R[t]$ for all $s,t\in S$. Since $R\subset S$ is subintegral, the isomorphism $R/M\cong S/N$ shows that $S=R+N$. 

Assume that $R\subset S$ is a $\Delta$-extension and let $x,y\in N\setminus R=N\setminus M$. Then $R\subset R[x]$ is minimal ramified by Theorem \ref{minimal} because $x^2\in M$ and $xM\subseteq M$. For the same reason, $R\subset R[y]$ is minimal ramified. Let $P$ be the maximal ideal of $R[x]$. If $R[x]=R[y]$, it follows that $x,y\in N\cap R[x]=P$, so that $xy\in P^2\subseteq M$. Assume $R[x]\neq R[y]$. Then, $xy\in R[x,y]=R[x]+R[y]=R+Rx+Ry$, so that $xy=a+bx+cy\ (*)$, with $a,b,c\in R$. Then, $a=xy-bx-cy\in R\cap N=M$. By $(*)$, we have $a+bx=y(x-c)\ (**)$. If $c\not\in M$, then $c\not\in N$, which leads to $c-x$ is a unit in $R[x]$ and in $S$. It follows that $y=(a+bx)(x-c)^{-1}$ in $S$. Since $(a+bx)(x-c)^{-1}\in R[x]$, we get $y\in R[x]$, a contradiction. Then, $c\in M$ and $cy\in M$ because  $M=(R:S)$. A similar proof shows that $bx\in M$, giving $xy\in M$, so that $N^2\subseteq M$.

Conversely, assume that  $N^2\subseteq M$. We claim that $R[s,t]=R[s]+R[t]$ for all $s,t\in S$. It is enough to prove that $st\in R[s]+R[t]$ for all $s,t\in S$. Since $S=R+N$, we can write $s=\lambda+x$ and $t=\mu+y$, with $\lambda,\mu\in R$ and $x,y\in N$. In particular, $xy\in N^2\subseteq M$. But, $st=\lambda\mu+\lambda y+\mu x+xy=\lambda(\mu+y)+\mu(\lambda+x)+x y-\lambda\mu =\lambda t+\mu s+x y-\lambda\mu \in R[s]+R[t]$. To conclude, $R\subset S$ is a  $\Delta$-extension.

We do not consider the case $(\delta)$ for the following reason. An extension verifying $(\delta )$ cannot be a $\Delta$-extension, because it is not catenarian: there exist two minimal extensions $R\subset R[x]$ and $R\subset R[y]$, with, for instance, $R\subset R[x]$ ramified, $R\subset R[y]$ inert and $M=(R:S)$ (see Proposition \ref{3.6}). In fact, $y\in S\setminus(R+N)$.
  \end{proof}

We saw in Proposition \ref{11.5} that in the seminormal infra-integral case, we deal with extensions of the form $R/M\subset(R/M)^3$. We are going to look at infra-integral extensions of the type $R\subset R^n$.

\begin{lemma} \label{11.9} Let $R$ be an Artinian ring and $n$ an integer with $n>1$. If $R\subset R^n$ is a $\Delta$-extension, then $n\leq 3$.
  \end{lemma} 
\begin{proof} Since $(R^n)_M=(R_M)^n$ for any maximal ideal $M$ of $R$, we may assume that $R$ is a local ring by Proposition \ref{9.02}. Set $S:=R^n$ and $T:={}_S^+R$. Then, $R\subset S$ is FCP
infra-integral by \cite [Proposition 1.4]{Pic 9}, with $|\mathrm{Max}(S)|=n$ and $\ell[T,S]=n-1$ by \cite[Lemma 5.4]{DPP2}. Using Corollary  \ref{9.182}, we get that $n\leq 3$. 
 \end{proof}

\begin{proposition} \label{11.10} Let $R$ be an  absolutely flat   ring. Then $R\subset R^n$ is a $\Delta$-extension if and only if  $n\leq 3$.
  \end{proposition} 
\begin{proof} By Proposition \ref{9.02}, we may assume that $R$ is local, so that $R$ is a field. Then \cite[Proposition 1.4]{Pic 9} shows that $R\subset R^n$ is an infra-integral seminormal FCP extension. By Proposition \ref{9.182}, $R\subset R^n$ is a $\Delta$-extension if and only if  $\ell[R,R^n]\leq 2$ if and only if  $n\leq 3$.
 \end{proof}
 
According to \cite[Proposition 1.4]{Pic 9}, when $R$ is not reduced and $n=3$, there is a subintegral part $R\subset{}_{R^3}^+R$ of $R\subset R^3$, so that we cannot use Corollary \ref{9.182}. We next give an example of a $\Delta$-extension $R\subset R^3$, where $R$ is an Artinian local and non-reduced ring (which is not a field). Moreover, $R\subset R^3$ has FCP and is infra-integral by the above reference. We leave the easy but tedious calculations of this example to the reader.
 
\begin{example} \label{11.11} Set $R:=(\mathbb Z/ 2\mathbb Z)[T]/(T^2)=(\mathbb Z/ 2\mathbb Z)[t]$, where $t$ is the class of $T$ in $R$. Then $R$ is an Artinian local ring which is not reduced and whose maximal ideal $M:=Rt\neq 0$ is such that  $M^2=0$. Moreover, $|R/M|=|\mathbb Z/ 2\mathbb Z|=2$. Let $\mathcal B:=\{e_1,e_2,e_3\}$ be the canonical basis of $R^3$. According to \cite[Proposition 2.8]{Pic 9},  $S:={}_{R^3}^+R=R+N=R(e_1+e_2+e_3)+Rte_1+Rte_2$ (for instance), where $N:=M\times M\times M=Rte_1+Rte_2+Rte_3$ is the maximal ideal of the local ring $S$. Since $|\mathrm{Max}(R^3)|=3$, we get $\ell[S,R^3]=2$ according to \cite[Lemma 5.4]{DPP2}; so that, $S\subset R^3$ is a $\Delta$-extension by Corollary \ref{9.182}, and $\ell[R,S]={\mathrm L}_R(N/M)$ by \cite[Lemma 5.4]{DPP2}. But $MN=0$ yields ${\mathrm L}_R(N/M)={\mathrm L}_{R/M}(N/M)=\dim_{R/M}[(M\times M\times M)/M]=2$ thanks to \cite[Corollary 2 of Proposition 24, page 66]{N}. It follows that $\ell[R,S]=2$ and $\ell[R,R^3]=4$, always by \cite[Lemma 5.4]{DPP2}. In particular, Theorem \ref{9.171} shows that $R\subset S$ is a  $\Delta$-extension. 
Hence, from Proposition \ref{9.23}, we deduce that $R\subset R^3$ is a $\Delta$-extension if and only if $\ell[T,UV]=2$, for each $T,U,V\in[R,R^3]$ such that $T\subset U$ is minimal ramified and $T\subset V$ is minimal decomposed. To find such extensions, we are going to determine all elements of $[R,R^3]$. 
  
For each $i=1,2,3$, set $x_i:=te_i$, so that $R[x_i]=R+Rte_i$ is a local ring with maximal ideal $N_i:=Rte_i+Rt(e_j+e_l)$, for $j,l\neq i$ and $ j\neq l$. Then, $\{R,R[x_1],R[x_2],R[x_3],S\}\subseteq [R,S]$, with the $R[x_i]$'s all distinct, so that $|[R,S]|>3$. Then, \cite[Theorem 6.1]{Pic 6} says that $|[R,S]|=5$, because $R\subset S$ is subintegral. In particular, $[R,S]=\{R,R[x_1],R[x_2],R[x_3],S\}$. Moreover, $|[S,R^3]|=5$ because $S\subset R^3$ is seminormal infra-integral. Indeed, $R\subset R^3$ is infra-integral. A short calculation shows that there does not exist any $y\in R^3$ such that $R\subset R[y]$ is minimal decomposed. By definition of $S={}_{R^3}^+R$, for each $i=1,2,3$, the only minimal ramified extension starting from $R[x_i]$ is $R[x_i]\subset S$. For each $i=1,2,3$, set 
$R_i:=R[e_i]=R+Re_i=R[x_i]+e_iR[x_i]$. We get that $R[x_i]\subset R_i$ is minimal decomposed and another calculation shows that this is the only minimal decomposed extension starting from $R[x_i]$. It follows that the maximal ideals of $R_i$ are $M_i:=M+Re_i=N_i+e_iR[x_i]$ and $M'_i:=M+R(1-e_i)$. By definition of $S$, there does not exist any $y\in R^3$ such that $S\subset S[y]$ is minimal ramified. Moreover, for each $i=1,2,3$, we have $S\subset SR_i\subset R^3$, with $SR_i=R+Re_i+Rte_j+Rte_l$, for $j,l\neq i$ and $ j\neq l$. Since $S\subset R^3$ is seminormal infra-integral, then $SR_i={}_{R^3}^+R_i$ because $R_i\subset SR_i$ is minimal ramified (use $SR_i=R_i[te_j]$ with $i\neq j$, for instance) and $ R_i\subset SR_i$ is the only minimal ramified extension starting from $R_i$. In particular, $[S,R^3]=\{S,SR_1,SR_2,SR_3,R^3\}$ because $|[S,R^3]|=5$ as we have seen just above. We claim that there does not exist any $y\in R^3$ such that $R_i\subset R_i[y]$ is minimal decomposed (if the contrary holds, then $R_i[y]\subset R^3$ is minimal ramified, contradicting \cite[Lemma 17]{DPP4}, because $MR^3=M\times M\times M$ is a radical ideal of $R^3$ containing $0=(R:R^3)$). To sum up, the only minimal extensions of $R$ are the $R[x_i]$, for a given $i$, the only minimal extensions of $R[x_i]$ are $R_i$ and $S$,  the only minimal extension of $R_i$ is $SR_i$, and the only minimal extensions of $S$ are the $SR_i$. At last, $SR_i\subset R^3$ is minimal for each $i$. Then, we get the following commutative diagram   corresponding to only one $i\in\{1,2,3\}$, where {\it r}, (resp. {\it d}) indicates a minimal ramified (resp. decomposed) extension:  
  $$\begin{matrix}
{} & {}            & {} &       {}         & R_i & {}             & {} & {} & {}\\
{} & {}            & {} & d\nearrow & {}    & r\searrow & {} & {} & {}\\
R & \overset{r}\rightarrow & R[x_i] & {} & {} & {} & SR_i & \overset{d}\rightarrow & R^3 \\
{} & {}           & {} & r\searrow  & {}    & d\nearrow & {} & {} & {}\\
{} & {}           & {} &      {}         & S    & {}               & {} & {} & {}
\end{matrix}$$
To conclude, the only extensions such that $T\subset U$ is minimal ramified and $T\subset V$ is minimal decomposed are gotten for $T=R[x_i],\ U=S$ and $V=R_i$ for each $i=1,2,3$. Since $UV=SR_i$ and $\ell[T,UV]=\ell[R[x_i],SR_i]=2$, we get that $R\subset R^3$ is a $\Delta$-extension by Proposition ~\ref{9.23} as we claimed before. 
 \end{example}
 
In the context of extensions of the form $R\subset R^n$, for a positive integer $n$, we now  consider extensions of the form $R\subset R[t]$, where $t$ is either  idempotent or nilpotent of index 2. We recall that a ring $R$ is called a {\it SPIR} if $R$ is a principal ideal ring with a nonzero prime ideal $M$ such that $M$ is nilpotent of index $p>0$. 
  
 \begin{proposition} \label{11.12} An extension $R\subset R[t]$, where $t$ is either idempotent or nilpotent of index 2, and such that $R$ is a SPIR, is an 
FIP
 $\Delta$-extension.
\end{proposition}

\begin{proof} Set $S:=R[t]$. If $t$ is either  idempotent or nilpotent of index 2, then $S$  is a finitely generated $R$-module. Moreover, $R/(R:S)$ is an Artinian ring because $R$ is a SPIR. Then, $R\subset S$ has FCP by \cite[Theorem 4.2]{DPP2}. In both cases, we can write $t^2-rt=0\ (*)$, with either $r=0$ or $r=1$. This implies $t^3-rt^2=0$, then $R\subset R[t]$ is an $t$-elementary extension \cite[Definition 1.1]{Pic} and \cite[Proposition 2.18]{Pic} shows that $R/(R:S)\subset S/(R:S)$ can be identified with $R/(R:S)\subset (R/(R:S))[X]/(X^2-\bar rX)$, where $\bar r$ is the class of $r$ in $R/(R:S)$ . But $R\subset S$ is a $\Delta$-extension if and only if $R/(R:S)\subset S/(R:S)$ is a $\Delta$-extension by Proposition \ref{9.02}. Then, there is no harm to assume that $(R:S)=0$, so that $S=R[t]=R+Rt\cong R[X]/(X^2-rX)$, with $r\in\{1,0\}$. 
   
Since $R$ is a SPIR, $(R,M)$ is an Artinian local ring and  $M+Rt$ is a maximal ideal of $S$ lying above $M$. If $M=0$, then $R$ is a field and either $S\cong R^2$ or $S\cong R[X]/(X^2)$. In both cases, $R\subset S$ is a minimal extension, decomposed when $r=1$ and ramified when $r=0$, and then a $\Delta$-extension. 
   
If $M\neq 0$, then $M$ is nilpotent and principal. Let $n$ be its index of nilpotency, so that $M^n=0$, with $M^{n-1}\neq 0$. Set $M:=Rx$. We begin by  showing a result that holds for both values of $r$. For $k\in\mathbb N_n$, set $R_k :=R+M^kS=R+Rx^kt$, which is obviously a local ring with maximal ideal $M_k:=M+M^kS=Rx+Rx^kt$. In particular, for $k<n$, we have $M_{k+1}=MR_k$ and $R_k=R_{k+1}[x^kt]$. Set $y_k:=x^kt$ and $\mathcal C:=\{R_k\}_{k\in\mathbb N_n}$; so that, $R_k=R_{k+1}[y_k]$. 

Next we show that  $\mathcal C$ is a finite maximal chain. Observe that $y_k^2=x^{2k}t^2=x^{2k}rt\in R_{k+1},\ M_{k+1}y_k=(Rx+Rx^{k+1}t)x^kt=Rx^{k+1}t+Rx^{2k+1}rt\subseteq M_{k+1}$ and 
 ${\mathrm L}_{R_{k+1}}(M_k/M_{k+1})={\mathrm L}_{R_{k+1}/M_{k+1}}(M_k/M_{k+1})=\dim_{R/M}(M_k/M_{k+1})=1$, since $M_k=M_{k+1}+Rx^kt$. It follows that $R_{k+1}\subset R_k$ is minimal ramified by \cite[Theorem 1.1]{Pic 9}. Therefore, $\mathcal C$ is a finite maximal chain and $R\subset R_1$ is subintegral. 
 
Now, we intend to prove that $[R,R_1]=\mathcal C$. Assume that the contrary holds; so that there exists some $U\in[R,R_1]\setminus\mathcal C$. Since $\{R_k\}_{k=1}^n$ is a decreasing chain going from $R_1$ to $R$, there exists some $k\in\mathbb N_{n-1}$ such that $U\subset R_k$, with $U\not\subseteq R_{k+1}$. Moreover, $R\subset S$ is an FCP extension. Since $U\neq R_k$, there exists $V\in[U,R_k]$ so that $V\subset R_k$ is minimal. In particular, $V$ is a local ring. Let $N$ be its maximal ideal. Then, $N:=(V:R_k)\in\mathrm{Max}(V)$, which is also an ideal of $R_k$. Since $R\subset V$ is subintegral, we get that $V=R+N$. Then, $M\subseteq N$ implies $M_{k+1}=MR_k\subseteq NR_k=N$, so that $R_{k+1}=R+M_{k+1}\subseteq R+N=V\subset R_k$, which leads to $V=R_{k+1}$. Then, $U\subseteq R_{k+1}$, a contradiction. Therefore, $[R,R_1]=\mathcal C$.

 To show that $R\subset S$ is a $\Delta$-extension, we split the proof in two cases, according to the value of $r$. 
 
 Assume that $r=1$. It follows that we are reduced to show that $R\subset R^2=:S$ is a $\Delta$-extension.
 
 According to \cite[Proposition 1.4]{Pic 9}, $R\subset R^2$ is an infra-integral FCP extension, and, setting $T:={}_S^+R$, \cite[Proposition 2.8]{Pic 9} says that $T=R+(M\times M)=R+MS=R_1$. Since $|\mathrm{Max}(S)|=2$, it follows that $R_1\subset S$ is minimal decomposed by \cite[Lemma 5.4]{DPP2}. Proposition \ref{9.23} asserts that $R\subset S$ is a $\Delta$-extension if and only if the following statements hold: $R\subset R_1$ is a $\Delta$-extension (1), $R_1\subset S$ is a $\Delta$-extension (2), and for each $W,U,V\in[R,S]$ such that $W\subset U$ is minimal ramified and $W\subset V$ is minimal decomposed, $\ell[W,UV]=2$ (3). Since $R_1\subset S$ is minimal, (2) is satisfied. We have proved above that $[R,R_1]=\mathcal C$ is a finite maximal chain, and then is a $\Delta$-extension. To show (3), we prove that there do not exist $W,U,V\in[R,S]$ such that $W\subset U$ is minimal ramified and $W\subset V$ is minimal decomposed. Assume that the contrary holds, then in particular, $2\leq|\mathrm{Max}(W)|+1=|\mathrm{Max}(V)|\leq|\mathrm{Max}(S)|=2$ implies that $|\mathrm{Max}(V)|=2$. Assume first that $W\not\in[R,R_1]\cup[R_1,S]$ and set $W':=W\cap R_1\in[R,R_1[$ with $W'\neq W$, so that $W'\subset W$ is seminormal infra-integral with $|\mathrm{Max}(W)|\geq 2$. But $2=|\mathrm{Max}(V)|=|\mathrm{Max}(W)|+1\geq 3$, a contradiction. It follows that $W\in[R,R_1]\cup[R_1,S]$. If $W\in[R_1,S]$, then $W\subset U\subseteq S$ together with $W\subset U$ minimal ramified leads to a contradiction. Then, $W\in[R,R_1]$, which implies $U\in[R,R_1]$. But $|\mathrm{Max}(V)|=2=|\mathrm{Max}(S)|$ implies that either $V=S$, which is impossible since $V\subset UV\subseteq S$, or $V\subset S$ is subintegral. In this case, there exists $V'\in[R,S]$ such that $V'\subset S$ is minimal ramified. Now, \cite[Lemma 17]{DPP4} gives that $MS$ is not a radical ideal of $S$, a contradiction with $MS=M\times M$. To conclude, (3) is an invalid statement and $R\subset S$ is a $\Delta$-extension when $r=1$. 
 Moreover, $R\subset S$ has FIP by \cite[Corollary 2.5]{Pic 9} because $R$ has finitely many ideals.

Assume now that $r=0$. Then $S=R[t]=R_1[t]$, where $t^2=0$ and $M_1t=Rxt+Rxt^2=Rxt\subseteq M_1$ show that $R_1\subset S$ is minimal ramified. In fact, $[R,S]=[R,R_1]\cup\{S\}$. It is enough to use the proof showing that $[R,R_1]$ is a maximal chain with $k=0$, setting $R_0:=S$ and $M_0:=M+Rt$. Since $R\subset S$ is chained, it is a $\Delta$-extension 
and has FIP,  because it has FCP.
 \end{proof}

 In the same spirit as the previous Proposition, we get the following result:

\begin{proposition} \label{11.13} Let $R\subset S$ be an FCP subintegral extension, where $(R,M)$ is a local Artinian ring, so that $(S,N)$ is also a local ring. Let $n>0$ be the index of nilpotency of $M$. Set $R_k:=R+M^kS$, for any $k\in\{0,\ldots,n\}$, with $R_0=S$ and $R_n=R$. Assume that $N^2\subseteq MS$ and that $[R,S]$ is pinched at $\{R_k\}_{k=1}^{n-1}$. Then, $R\subset S$  is a $\Delta$-extension.
\end{proposition}

\begin{proof} For each $k\in\{1,\ldots,n\}$, set $M_k:=M+M^kS$ and $M_0:=N$. Mimicking the proof of Proposition \ref{11.12}, we get that $(R_k,M_k)$ is a local ring such that $M_{k+1}=MR_k=(R_{k+1}:R_k)\ (*)$ for each $k\in\{0,\ldots,n-1\}$. Since $R\subset S$ is subintegral, so are each  $R_{k+1}\subset R_k$. Then $R_k=R+M_k=R+M^kS$. We begin to show that each $R_{k+1}\subset R_k$ is pointwise minimal. For  $k\in\{1,\ldots,n-1\}$, let $x=a+y\in R_k\setminus R_{k+1},\ a\in R,\ y\in M^kS$, so that $R_{k+1}[x]=R_{k+1}[y]$. For $k=0$, we choose $y\in N\setminus MS$. Then, $M_{k+1}y\subseteq M_{k+1}$ by $(*)$, which also holds for $k=0$. Moreover, $y^2\in M_k^2=M^2+M^{k+1}S+M^{2k}S\subseteq M_{k+1}$ shows that $R_{k+1}\subset R_{k+1}[y]$ is minimal, so that $R_{k+1}\subset R_k$ is pointwise minimal. This also holds for $k=0$ since $N^2\subseteq MS$. In particular, this shows that $R_1\subset S$ is a $\Delta$-extension by Proposition \ref{11.5}. The same Proposition shows that $R_{k+1}\subset R_k$ is a $\Delta$-extension because $M_k^2\subseteq M_{k+1}$ as we have just seen. Then, $R\subset S$  is a $\Delta$-extension because $[R,S]$ is pinched at $\{R_k\}_{k=1}^{n-1}$ according to Corollary \ref{9.03}.
\end{proof}

 \begin{proposition}\label{11.14} An FIP subintegral extension $k\subset S$ over the field $k$ is  a $\Delta$-extension  if  either (1) $|k|=\infty$ or (2) $k\subset S$ is chained.
\end{proposition}
\begin{proof} We use \cite[Theorem 3.8]{ADM} together  with the fact that $k\subset S$ is FIP subintegral. This last condition implies that either (a) : $|k|<\infty$ with $S$ a finite dimensional vector-space such that $k\subset S$ is  subintegral, or (b) $|k|=\infty$ with $S=k[\alpha]$ for some $\alpha\in S$ which satisfies $\alpha^3=0$.  

If (2) holds, then $k\subset S$ is a $\Delta$-extension by Proposition \ref{1.013}.

Assume now (1), that is $|k|=\infty$. Since $k\subset S$ is  subintegral, $S$ is a local ring with maximal ideal $M:=k\alpha+k\alpha^2$. If $\alpha^2=0$, it follows that $k\subset S$ is minimal ramified, and then a $\Delta$-extension. If $\alpha^2\neq 0$, \cite[Lemma 5.4]{DPP2} shows that $\ell[k,S]={\mathrm L}_k(M)=\dim_k(M)=2$. Hence  $k\subset S$ is  a $\Delta$-extension according to Theorem \ref{9.171}.
\end{proof} 

 \begin{remark}\label{11.141} Even if $|k|<\infty$, then, $|k(X)|=\infty$ and we may use 
 (1) of Proposition \ref{11.14}
  for the extension $k(X)\subset S(X)$.  Because of Proposition \ref{8.14}, then  $k\subset S$ is a $\Delta$-extension if so is $k(X)\subset S(X)$. But, in order to use (1), we need that $k(X)\subset S(X)$ has FIP, this last property being equivalent to $k\subset S$ is an FIP chained extension according to \cite[Theorem 4.2]{Pic 4}.
\end{remark} 

\begin{corollary}\label{11.15} Let $R:=\prod_{i=1}^nk_i$ be a product of infinite fields  and $R\subset S$ be an FIP subintegral extension. Then, $R\subset S$ is a $\Delta$-extension.
\end{corollary}
\begin{proof}  Proposition \ref{5.13} says that for each $i\in\mathbb N_n$, there exists ring extensions $k_i\subseteq S_i$ such that $S\cong \prod_{i=1}^n S_i$. Moreover $R \subseteq S$ is a subintegral $\Delta$-extension if and only if so is $k_i\subseteq S_i$ for each $i\in\mathbb N_n$ (see the proof of Proposition  \ref{5.13}). Conclude with Proposition~\ref{11.14}.
 \end{proof}  

Here is an example of an infra-integral $\Delta$-extension of number field orders whose length is $>2$. Its $\Delta$-property is proved by checking the hypotheses of Proposition \ref{9.23} and in particular the  condition (2) of Proposition \ref{9.23}. 
 
 \begin{example}\label{13} In \cite[Example 3.7 (5)]{FCA}, El Fadil, Chillali and Akharraz consider the quartic number field $K$ defined by the irreducible polynomial $X^4+22X+66$. Let $S$ be the ring of integers of $K$. It is shown in this example that $3S=P_1P_2^3$, where $P_1$ and $P_2$ are the maximal ideals of $S$ lying above $3\mathbb Z$. Set $R:=\mathbb Z+3S$. We are going to prove that $R\subset S$ is a $\Delta$-extension. 
 
Since $[K:\mathbb Q]=4$, the fundamental formula \cite[Theorem 1, page 193]{Ri} gives $4=\sum_{i=1}^ge_if_i$, where $g$ is the decomposition number of $3\mathbb Z$ in the extension $\mathbb Q\subset K,\ e_i$ is the ramification index of $P_i$ and $f_i$ is the inertial degree of $P_i$. It follows that $g=2,\ e_1=1,\ e_2=3$ and $f_i=1$ for each $i$. Observe that $S/P_i\cong k$ for each $i$, where $k:=\mathbb Z/3\mathbb Z$. In particular, $R\subset S$ is an infra-integral extension because  $3S=(R:S)$ is a maximal ideal of $R$ and $R/3S\cong\mathbb Z/3\mathbb Z=k\cong S/P_i$ for each $i$, where $P_1$ and $P_2$ are the only maximal ideals of $S$ containing $(R:S)=3S$. Set $W:={}_S^+R$. Then $W\neq S$ because $\mathrm{V}_S(3S)=\{P_1,P_2\}$ and $3S\in\mathrm{Max}(R)$, so that $R\subset S$ is not subintegral. Since $|\mathrm{V}_S(3S)|=2$, it follows that $W\subset S$ is minimal decomposed according to Theorem \ref{minimal} and Proposition \ref{3.5} and then a $\Delta$-extension. Hence, the second part of condition (1) of Proposition \ref{9.23} is satisfied. Now, $3S=P_1P_2^3\subseteq(W:S)\subseteq P_1P_2$ implies $(W:S)=P_1P_2$ by the same reference. Since $(W:S)\in\mathrm{Max}(W)$ with $\mathbb Z/3\mathbb Z\cong(\mathbb Z+P_1P_2)/P_1P_2\subseteq W/P_1P_2\cong S/P_i\cong\mathbb Z/3\mathbb Z$ for each $i=1,2$, this  shows that $W=\mathbb Z+P_1P_2$. Moreover, $3S=(R:W)$. Let $N$ be the maximal ideal of the local ring $W/3S$.  
  
Because $(R:S)=3S$, \cite[Lemma 5.4]{DPP2} gives $\ell[R,W]=\ell[R/3S,W/3S]$

\noindent $={\mathrm L}_k(N)=\dim_k(N)< \dim_k(W/3S)< \dim_k(S/3S)=4$, so that $\ell[R,W]\leq 2$. It follows that $R\subset W$ is a $\Delta$-extension by Theorem  \ref{9.171} and condition (1) of Proposition \ref{9.23} is satisfied. In particular, $\ell[R,S]\leq 3$ since $R\subset S$ is infra-integral. We show that condition (2) of Proposition \ref{9.23} is also satisfied.  
  
Set $U_1:=R+P_1P_2^2,\ V_1:=R+P_2^3,\ V_2:=R+P_2^2$ and $U_2:=R+P_1P_2$. We have the following commutative diagram 
 with $W=U_2$ so that $\ell[R,W]= 2$ and $\ell[R,S]=3$:
$$\begin{matrix}
{} &       {}       & U_1 &      \to     & U_2 &      {}       & {} \\
{} & \nearrow &   {}   &       {}      &    {}  & \searrow & {} \\
R &       {}      &   {}   & \searrow &    {}  &      {}       & S \\
{} &\searrow  &   {}   &      {}       &    {}  & \nearrow & {} \\
{} &     {}        & V_1 &      \to     & V_2 &       {}       & {} 
\end{matrix}$$
Since $(R:S)=P_1P_2^3$, for any $T\in[R,S]$, we have $P_1P_2^3\subseteq (T:S)\ (*)$, so that $P_1$ and $P_2$ are the only maximal ideals of $S$ that may contain $(T:S)$. Moreover, $(T:S)=P_1^{\alpha}P_2^{\beta}$, for some $(\alpha,\beta)\in\{0,1\}\times\{0,1,2,3\}$, because $S$ is a Dedekind domain. In particular, if $T\subset S$ is minimal, it is either ramified, and in this case, $M^2\subseteq (T:S)\subset M$ for some maximal ideal $M$ of $S$. This leads to $(T:S)=P_2^2$ and $T=V_2$. If $T\subset S$ is decomposed, the only possible case is $(T:S)=P_1P_2$ and $T=W=U_2$. 
 
 Let $T,U,V\in[R,S]$ be such that $T\subset U$ is minimal ramified and $T\subset V$ is minimal decomposed. Since $\ell[R,S]=3$, we may have $\ell[T,UV]>2$ only if $T=R$ and $UV=S$. We are going to  show that $V=V_1$.  Of course, the diagram shows that $R\subset V_1$ is minimal and $P_2^3=(V_1:S)$ is a maximal ideal of $V_1$. Since $P_2$ is the only maximal ideal of $S$ lying above $(V_1:S)$, it follows that $V_1\subset S$ is subintegral. Now $R\subset V_1$ is minimal decomposed, because if we suppose that the contrary holds, then $R\subset S$ is subintegral, a contradiction. Assume that there is another $V'\in[R,S]$ such that $R\subset V'$ is minimal decomposed. There would be in $V_1V'$, 3 maximal ideals lying above $3S$ because $\ell[R,V_1V']\geq 2$ in this case, with $R\subset V_1V'$ seminormal infra-integral, a contradiction since only 2 maximal ideals of $S$ lie above $3S$. Then, $V_1$ is the only $V\in[R,S]$ such that $R\subset V$ is minimal decomposed. Now, let $U\in[R,S]$ be such that $R\subset U$ is minimal ramified, so that there is only one maximal ideal $N$ of $U$ lying above $3S=(R:S)$. Since $P_1$ and $P_2$ lie above $3S$ in $R$, they both lie above $N$ in $U$, so that $N\subseteq P_1P_2$. It follows that $NP_2^3\subseteq P_1P_2^4\subset P_1P_2^3$, which shows that $\ell[R,UV_1]=2$ according to Proposition \ref{3.6}.  Since condition (2) of Proposition \ref{9.23} is also satisfied, $R\subset S$ is a $\Delta$-extension.
\end{example}

We now introduce  a property linked to $\Delta$-extensions and to \cite{GH} (see Proposition \ref{9.0}). Let $R\subset S$ be a ring extension. We say that $R\subset S$ is a $\delta$-extension
    if $R[x]+R[y]=R[x+y]$ for any $x,y\in S$ such that $R[x]\neq R[y]$. 
    
\begin{proposition} \label{11.21} Let $R\subset S$ be a ring extension.
\begin{enumerate}
\item $R\subset S$ is a $\delta$-extension if and only if $\sum_{i=1}^nR[x_i]=R[\sum_{i=1}^nx_i]$ for any $x_1,\ldots,x_n\in S$ and any integer $n$ such that $R[x_i]\neq\sum_{j\in I}R[x_j]$, for any $I\subseteq {\mathbb N}_n\setminus\{i\}$ and for any $i\in{\mathbb N}_n$.
\item If $R\subset S$ is a $\delta$-extension, then $R\subset S$ is a $\Delta$-extension.
\item If $R\subset S$ is an FCP  $\delta$-extension, then  $R\subset S$ is simple.
\end{enumerate} 
\end{proposition}

\begin{proof} (1) One implication is obvious. Conversely, assume that $R\subset S$ is a $\delta$-extension, that is $R[x]+R[y]=R[x+y]$ for any $x,y\in S$ such that $R[x]\neq R[y]$, and let $x_1,\ldots,x_n\in S$ be such that $R[x_i]\not\subseteq\sum_{j\in I}R[x_j]$, for any $I\subseteq {\mathbb N}_n\setminus\{i\}$ and for any $i\in{\mathbb N}_n$. We show that $\sum_{i=1}^nR[x_i]=R[\sum_{i=1}^nx_i]$ by induction on $n$. The induction hypothesis is obviously satisfied for $n=2$. Let $n>2$ and set $y:=\sum_{i=1}^{n-1}x_i$, so that $\sum_{i=1}^nx_i=y+x_n$.  Assume that the induction hypothesis holds for $n-1$. It follows that $R[y]=\sum_{i=1}^{n-1}R[x_i]$, with $R[x_n]\not\subseteq R[y]$. Then, $R[\sum_{i=1}^nx_i]=R[y+x_n]=R[y]+R[x_n]$ by the hypothesis, which leads to $R[\sum_{i=1}^nx_i]=\sum_{i=1}^{n-1}R[x_i]+R[x_n]=\sum_{i=1}^nR[x_i]$. Then, it holds for any $n$. 

(2) If $R\subset S$ is a $\delta$-extension, then $R[x]+R[y]=R[x+y]$ for any $x,y\in S$ such that $R[x]\neq R[y]$, so that $R[x]+R[y]\in [R,S]$ for any $x,y\in S$, which shows that $R\subset S$ is a $\Delta$-extension by Proposition \ref{9.0} since $x,y\in R[x]+R[y]$ implies $R[x,y]\subseteq R[x]+R[y]\subseteq R[x,y]$.

(3) Recall that an extension $R\subset S$ is called {\it strongly affine} if each element of $[R,S]$ is a finite-type $R$-algebra. Assume that $R\subset S$ is an FCP $\delta$-extension. Then, $R\subset S$ is strongly affine by \cite[Proposition 3.12]{DPP2}, so that $S=R[x_1,\ldots,x_n]$ for some $x_1,\ldots,x_n\in S$. We can choose the $x_i$'s as a minimal generating set, so that 
$R[x_i]\not\subseteq\sum_{j\in I}R[x_j]$, for any $I\subseteq{\mathbb N}_n\setminus\{i\}$ and for any $i\in{\mathbb N}_n$. Moreover, $R\subset S$ is a $\Delta$-extension by (2). Then, we get $S=\sum_{i=1}^nR[x_i]$ by Proposition \ref{9.0}, so that $S=R[\sum_{i=1}^nx_i]$ and $R\subset S$ is simple.
\end{proof}

\begin{proposition} \label{11.22}  A chained extension   is a $\delta$-extension. 
\end{proposition}

\begin{proof}
Let $R\subset S$ be a chained extension, and let $x,y\in S$ be such that $R[x]\neq R[y]$. Since $R[x]$ and $ R[y]$ are comparable, assume $R[x]\subset R[y]$, so that $x\in R[y]$ which implies $R[x+y]\subseteq R[y]$. Moreover, $R[x+y]$ and $ R[x]$ are comparable. If $R[x+y]\subseteq R[x]$, then $x+y\in R[x]$ which gives $y\in R[x]$, a contradiction. It follows that $R[x]\subset R[x+y]$, and then $x\in R[x+y]$, which gives $y\in R[x+y]$. Then, $R[x+y]=R[y]=R[y]+ R[x]$ and $R\subset S$ is a $\delta$-extension.
\end{proof}

\begin{corollary} \label{11.23} A Pr\"ufer extension is a $\delta$-extension.
\end{corollary}

\begin{proof}
Let $R\subset S$ be a Pr\"ufer extension, and so is $R_M\subset S_M$ for any $M\in\mathrm{MSupp}(S/R)$. Let  $M\in\mathrm{MSupp}(S/R)$. 
 According to \cite[Proposition 1.2]{Pic 5},  $R_M\subset S_M$ is chained. Indeed, $R_M\subset S_M$ is Pr\"ufer for each $M\in\mathrm{Supp}(S/R)$ by \cite[Proposition 1.1]{Pic 5}, and since $R_M$ is local, there exists $P\in\mathrm{Spec}(R_M)$ such that $S_M=(R_M)_P,\ P=PS_M$, with $R_M/P$ a valuation domain with quotient field $S_M/P$. Then, $R_M/P\subset S_M/P$ is chained, and so is $R_M\subset S_M$. It follows that $R_M\subset S_M$ is  a $\delta$-extension by Proposition \ref{11.22}. Let $x,y\in S$. Then, $R_M[x]+R_M[y]=R_M[x+y]$. Since this holds for any $M\in\mathrm{MSupp}(S/R)$, we get that $R[x]+R[y]=R[x+y]$, so that $R\subset S$ is  a $\delta$-extension.
\end{proof}

\begin{example} \label{11.24}(1) Let $R\subset T$ and $R\subset U$ be two minimal extensions such that  $S:=TU$ exists. If $R\subset S$ satisfies one of the two following conditions, then $R\subset S$   is  a $\delta$-extension:

(a) $\mathcal{C}(R,T)\neq\mathcal{C}(R,U)$.

(b) $R\subset T$ and $R\subset U$ are two minimal infra-integral extensions of different types such that $\ell[R,S]=2$.

In both cases, we can set $T=R[x]$ and $U=R[y]$.

If case (a) holds, let $M:=\mathcal{C}(R,T)$ and $N:=\mathcal{C}(R,U)$. Then, $U_M=R_M$ and $T_N=R_N$ imply $(T+U)_M=T_M=S_M$ and $(T+U)_N=U_N=S_N$, so that $T+U=S$. Since $x+y\not\in T,U$, we have $S=R[x,y]=R[x+y]$ because $[R,S]=\{R,T,U,S\}$ by Proposition \ref{3.6}. Then, $R[x]+R[y]=R[x+y]$. Finally, any $z\in S$ is such that $R[z]\in \{R,T,U,S\}$, so that $R\subset S$ is  a $\delta$-extension.

If case (b) holds, we can  assume $M:=\mathcal{C}(R,T)=\mathcal{C}(R,U)$ (if not, then  (a) holds). Now, $\ell[R,S]=2$ implies $[R,S]=\{R,T,U,S\}$ by  \cite[Theorem 6.1 (5)]{Pic 6} because $|[R,S]|\geq 4$. We still have $S=R[x+y]$ because $x+y\not\in R,T,U$. But, in this case, we can choose $xy\in M$ and $S=R+Rx+Ry$ by  Proposition \ref{3.6}. The end of the  proof of case (a) still holds, giving that  $R\subset S$ is  a $\delta$-extension.
 
(2) Proposition \ref{11.21} asserts that a $\delta$-extension is a $\Delta$-extension. The converse does not hold. Let $R\subset S$ be a non-simple FCP extension. Then, $R\subset S$ is not a $\delta$-extension by Proposition \ref{11.21}. Take $R:=\mathbb Z/2\mathbb Z$ and $S:=R^3$. According to Proposition \ref{11.5}, $R\subset S$ is a pointwise minimal $\Delta$-extension which is not minimal. In particular, $R\subset S$ is not simple by the definition of a pointwise minimal extension. Then, $R\subset S$ is a $\Delta$-extension which is not a $\delta$-extension.
\end{example}

\end{document}